\documentclass[letterpaper]{amsart}
\usepackage{amssymb}
\usepackage[all]{xy}

\newcommand{\C}{\mathbb{C}}
\newcommand{\CP}{\mathbb{CP}}
\newcommand{\D}{\mathbb{D}}
\newcommand{\Z}{\mathbb{Z}}
\newcommand{\p}{{}^p\!}
\newcommand{\barp}{{}^{\bar p}\!}
\newcommand{\barGx}{\overline{Gx}}
\newcommand{\barGy}{\overline{Gy}}

\newcommand{\cA}{\mathcal{A}}
\newcommand{\cC}{\mathcal{C}}
\newcommand{\cD}{\mathcal{D}}
\newcommand{\cF}{\mathcal{F}}
\newcommand{\cG}{\mathcal{G}}
\newcommand{\cH}{\mathcal{H}}
\newcommand{\cI}{\mathcal{I}}

\newcommand{\cL}{\mathcal{L}}
\newcommand{\cO}{\mathcal{O}}
\newcommand{\cP}{\mathcal{P}}
\newcommand{\cQ}{\mathcal{Q}}

\newcommand{\uZ}{\underline{Z}}

\newcommand{\cIC}{\mathcal{IC}}

\newcommand{\cloc}[3]{\cC^{\mathrm{loc}}_G(#1,#2)_{\le #3}}
\newcommand{\czloc}[1]{\cloc{X}{Z}{#1}}

\newcommand{\csupp}[2]{\cC^{\mathrm{supp}}_G(#1,#2)}
\newcommand{\czsupp}{\csupp{X}{Z}}
\newcommand{\qsupp}[2]{\cQ^{\mathrm{supp}}_G(#1,#2)}
\newcommand{\qzsupp}{\qsupp{X}{Z}}

\newcommand{\cn}[1]{\cC(#1)}
\newcommand{\cg}[1]{\cC_G(#1)}
\newcommand{\cgl}[2]{\cC_G(#1)_{\le #2}}
\newcommand{\cgg}[2]{\cC_G(#1)_{\ge #2}}

\newcommand{\tcgg}[2]{\tilde\cC_G(#1)_{\ge #2}}
\newcommand{\qn}[1]{\cQ(#1)}
\newcommand{\qg}[1]{\cQ_G(#1)}

\newcommand{\dg}[1]{\cD_G(#1)}
\newcommand{\dgb}[1]{\cD_G^{\mathrm{b}}(#1)}
\newcommand{\dgp}[1]{\cD^+_G(#1)}
\newcommand{\dgm}[1]{\cD^-_G(#1)}

\newcommand{\Db}{D^{\mathrm{b}}}
\newcommand{\dgml}[2]{\cD^-_G(#1)^{\le #2}}
\newcommand{\dgpg}[2]{\cD^+_G(#1)^{\ge #2}}
\newcommand{\dgbl}[2]{\cD^{\mathrm{b}}_G(#1)^{\le #2}}
\newcommand{\dgbg}[2]{\cD^{\mathrm{b}}_G(#1)^{\ge #2}}
\newcommand{\dsupp}[2]{\cD^{\mathrm{supp}}_G(#1,#2)}
\newcommand{\dzsupp}{\dsupp{X}{Z}}

\newcommand{\sm}[1]{\mathcal{M}(#1)}
\newcommand{\smpm}[1]{\mathcal{M}^{!*}(#1)}

\newcommand{\Oxmod}{\text{$\cO_x$-$\mathfrak{mod}$}}
\newcommand{\red}{{\mathrm{red}}}
\newcommand{\gen}{{\mathrm{gen}}}
\newcommand{\st}{\diamond}
\newcommand{\bt}{\blacktriangle}
\newcommand{\hto}{\hookrightarrow}
\newcommand{\ssm}{\smallsetminus}
\newcommand{\id}{\mathrm{id}}
\newcommand{\Forg}{\mathsf{F}}
\newcommand{\Av}{\mathsf{Av}}

\newcommand{\sref}[1]{{\rm(S\ref{#1})}}
\newcommand{\aref}[1]{{\rm(A\ref{#1})}}
\newcommand{\fref}[1]{{\rm(F\ref{#1})}}

\DeclareMathOperator{\im}{im}
\DeclareMathOperator{\alt}{alt}
\DeclareMathOperator{\ev}{ev}
\DeclareMathOperator{\cod}{cod}
\DeclareMathOperator{\scod}{scod}
\DeclareMathOperator{\codim}{codim}
\DeclareMathOperator{\step}{step}
\DeclareMathOperator{\Hom}{Hom}
\DeclareMathOperator{\Ext}{Ext}
\DeclareMathOperator{\RHom}{\mathit{R}Hom}
\DeclareMathOperator{\cHom}{\mathcal{H}\mathit{om}}
\DeclareMathOperator{\cRHom}{\mathit{R}\mathcal{H}\mathit{om}}
\DeclareMathOperator{\gHom}{\mathbf{Hom}}
\DeclareMathOperator{\gExt}{\mathbf{Ext}}

\DeclareMathOperator{\Spec}{Spec}
\DeclareMathOperator{\Proj}{Proj}  

\newtheorem{thm}{Theorem}[section]
\newtheorem{lem}[thm]{Lemma}
\newtheorem{prop}[thm]{Proposition}
\newtheorem{cor}[thm]{Corollary}

\theoremstyle{definition}
\newtheorem{defn}[thm]{Definition}

\theoremstyle{remark}
\newtheorem{rmk}[thm]{Remark}

\numberwithin{equation}{section}

\renewcommand{\flat}{{!}}

\title[Staggered $t$-structures]{Staggered $\boldsymbol{t}$-structures on derived categories of equivariant coherent sheaves}
\author{Pramod N. Achar}
\thanks{The author was partly supported by NSF grant DMS-0500873.}
\address{Department of Mathematics, Louisiana State University, Baton Rouge, LA \ 70803}
\email{pramod@math.lsu.edu}

\begin{document}

\begin{abstract}
Let $X$ be a scheme, and let $G$ be an affine group scheme acting on $X$. 
Under reasonable hypotheses on $X$ and $G$, we construct a $t$-structure on
the derived category of $G$-equivariant coherent sheaves that in many ways
resembles the perverse coherent $t$-structure, but which incorporates
additional information from the $G$-action.  Under certain circumstances, the heart of this $t$-structure, called the ``staggered $t$-structure,'' is a finite-length category, and its simple objects are particularly easy to describe.  We also exhibit two small examples in which the staggered $t$-structure is
better-behaved than the perverse coherent $t$-structure.
\end{abstract}

\maketitle

\section{Introduction}
\label{sect:intro}

Let $X$ be a scheme (say, over a field), and let $G$ be an affine group
scheme acting on $X$.  \emph{Perverse coherent sheaves} are the objects in
the heart of a certain nonstandard $t$-structure on the bounded derived
category of equivariant coherent sheaves on $X$, first introduced by
Deligne, but now more widely known from an exposition by
Bezrukavnikov~\cite{bez:pc}.  One key feature of this category, which we 
denote $\cP(X)$, is that it interacts well with Grothendieck--Serre
duality, just as the much better-known category of perverse (constructible)
sheaves interacts well with Poincar\'e--Verdier duality.  Assume now that
$G$ acts on $X$ with finitely many orbits.  If the boundary of each orbit
has codimension at least $2$ in the closure of that orbit, then $\cP(X)$
also enjoys the following remarkable properties:
\begin{itemize}
\item There are ``middle-extension functors'' that associate to any
irreducible equivariant vector bundle on an orbit a certain simple
perverse coherent sheaf, supported on the closure of that orbit.
\item Every simple perverse coherent sheaf arises in this way.
\item $\cP(X)$ is a finite-length category; {\it i.e.}, every object has finite
length.
\end{itemize}
(If we replace ``vector bundle'' by ``local system'' and delete the word
``coherent'' throughout, we get a list of well-known and important
properties of perverse sheaves.)  A particularly nice situation occurs
when all orbits have even codimension: then $\cP(X)$ can be made
self-dual. A key example of the latter is that in which $G$ is a reductive
algebraic group over an algebraically closed field, and $X$ is its
unipotent variety (see, for instance,~\cite{a} and~\cite{bez:qe}).

However, there are many common examples in representation theory in which
the codimension condition does not hold: for instance, the orbits of a
reductive algebraic group on the Springer resolution of its unipotent
variety, or the orbits of a Borel subgroup on the flag variety.  To
understand what goes wrong with perverse coherent sheaves on such
varieties, let us consider the self-dual case.  Roughly
speaking, the reason for the even codimension condition is this: the local
cohomology of a (suitably normalized) dualizing complex on a given orbit
$C$ is concentrated in degree $\codim C$, so in order to define a self-dual
$t$-structure, we must be able to put coherent sheaves on $C$ in degree
$\frac{1}{2}\codim C$.  (For comparison, the dualizing complex for
Poincar\'e--Verdier duality on a topologically stratified space
(resp.~algebraic variety) measures the real codimension (resp.~\emph{twice}
the algebraic codimension) of a given stratum, so in order to have a
self-dual category of perverse constructible sheaves, all strata must have
even real codimension (resp.~there is no restriction on codimensions).)

The main idea of this paper is to
use extra information coming from the equivariant structure on each orbit
to ``stretch out'' our interpretation of what the dualizing complex is
measuring.  Suppose that on each $G$-orbit, each irreducible equivariant
vector bundle $\cL$ is assigned an integer invariant, called its
\emph{step}, satisfying certain compatibility conditions with tensor
products, Hom-sheaves, {\it etc}.  Now, given $d \in \Z$, consider the
shifted object $\cL[d]$ in the derived category: this is concentrated in
degree $-d$.  Let us declare the \emph{staggered degree} of $\cL[d]$ to be
$-d + \step\cL$.  We can then try to refine the construction of the
perverse coherent $t$-structure in a way that keeps track of the
interaction between Grothendieck--Serre duality and staggered degrees.  

We show in this paper that this idea can indeed be carried out: we
construct a ``staggered $t$-structure'' on the bounded derived category of
equivariant coherent sheaves.  This $t$-structure resembles the perverse
coherent $t$-structure in many ways, most significantly in terms of its
interaction with Grothendieck--Serre duality.  Staggered sheaves on a
single $G$-orbit are objects in the derived category with a certain fixed
staggered degree, just as perverse coherent sheaves on a single $G$-orbit
are objects concentrated in a certain fixed degree in the derived category.

Moreover, we can define the \emph{staggered codimension} of closed
$G$-invariant subschemes in terms of staggered degrees of the dualizing
complex.  It turns out that if $G$ acts on $X$ with
finitely many orbits, each of which has a boundary of staggered
codimension at least $2$, then the most desirable properties of perverse
sheaves hold: every object in the heart of the staggered
$t$-structure has finite length, and every simple object arises by
middle-extension of an irreducible vector bundle on a single orbit.

Of course, for this construction to be useful, there must be examples in
which the category of staggered sheaves is finite-length but the
category of perverse coherent sheaves is not.  We exhibit two small
examples in which this is the case: that of $\C^\times$-equivariant sheaves on $\C$, and that of
$B$-equivariant sheaves on $\CP^1$, where $B$ is a Borel subgroup of
$SL_2(\C)$.  ($\CP^1$ is, of course, the flag variety for $SL_2(\C)$.) 
Both of these spaces consist of two orbits, of dimensions $0$ and $1$; the
category of perverse coherent sheaves in these cases is not finite-length. 
However, in both cases, the $0$-dimensional orbit turns out to have
staggered codimension at least $2$, so the category of staggered sheaves is
finite-length.  It is hoped that similar results turn out to be true in broader
classes of examples.  It would be interesting to compare the staggered
$t$-structure to other $t$-structures known to have finite-length hearts, such
as the ``exotic'' $t$-structure for equivariant coherent sheaves on the
Springer resolution, described in~\cite{bez:ct}.

\medskip 

For the sake of aesthetics and generality, most of the paper
(Sections~\ref{sect:prelim}--\ref{sect:stag-dt} and part of
Section~\ref{sect:ic}) is written without the assumption that $G$ acts
with finitely many orbits, although the author is not aware of an
interesting example in which that is the case.  Section~\ref{sect:prelim}
lists notation and collects some basic facts about coherent sheaves and
abelian categories.  In Section~\ref{sect:s-struc}, we axiomatize the
additional information that our coherent sheaves should carry with the
notion of ``$s$-structure.''  We actually need a global version of this
information, not just one instance on each orbit.  To that end,
Sections~\ref{sect:closed} and~\ref{sect:glue} are aimed at proving a
``gluing theorem'' for $s$-structures.  Next, in
Section~\ref{sect:duality}, we study the interaction between $s$-structures
and Grothendieck--Serre duality.  In Sections~\ref{sect:stag-orth} and~\ref{sect:stag-dt}, we finally
define the staggered $t$-structure, and prove that it actually is a
$t$-structure.  Section~\ref{sect:ic} gives the construction of the
middle-extension functor and a criterion for the heart of the staggered
$t$-structure to be finite-length.  In Section~\ref{sect:finite}, we prove a
theorem that helps in calculating examples by simplifying the checking of
the (rather long) list of axioms for an $s$-structure.  Finally,
Section~\ref{sect:examples} contains the two examples mentioned above.

\subsection*{Acknowledgements}

This work was largely inspired by a question I was asked by E.~Vasserot.  I
would like to thank D.~Sage, M.~Schlichting, and E.~Vasserot for helpful
and enlightening conversations.  I am grateful to the referee for help with the arguments in Section~\ref{subsect:equiv-der}.

\bigskip

{\em Added in revision.}  Since this paper first appeared in preprint form in September 2007, considerable progress has been made in the theory of staggered sheaves.  Two rich classes of examples are now available: flag varieties~\cite{as:flag} and toric varieties~\cite{t}.  Some key results on simple objects are now known to hold in greater generality than proved here~\cite{at:bs}.  With respect to a suitable filtration of the derived category~\cite{at:bs}, staggered sheaves  obey ``purity'' and ``decomposition'' theorems~\cite{at:pd}, analogous to results of~\cite{bbd} for $\ell$-adic mixed perverse sheaves.  Finally, every simple staggered sheaf admits both a projective cover and a ``standard'' cover~\cite{a:qhp}, the latter being analogous to a Verma module in category $\cO$. 

\section{Preliminaries}
\label{sect:prelim}

\subsection{Notation and assumptions}

Let $X$ be a scheme of finite type over a noetherian base scheme admitting
a dualizing complex in the sense of~\cite[Chap.~V]{har}, and let $G$ be an affine group scheme over the same
base, acting on $X$.  Let $\cn X$ (resp.~$\qn X$) denote the category of coherent (resp.~quasicoherent) sheaves on $X$, and let $\cg X$ (resp.~$\qg X$) denote the category of
$G$-equivariant coherent (resp.~quasicoherent) sheaves on $X$.  
We also assume, following~\cite{bez:pc}, that $G$ is flat, of finite type, and Gorenstein over the base scheme, and that both $\cn X$ and $\cg X$ have enough locally free objects.

The terms ``subscheme,'' ``sheaf,'' and ``vector bundle'' should always be understood to mean ``$G$-invariant
subscheme,'' ``$G$-equivariant coherent sheaf,'' and ``$G$-equivariant vector bundle,'' unless explicitly specified otherwise.  Similarly, the term
``irreducible'' should always be understood in the $G$-invariant sense: a
scheme is irreducible if it is not a union of two proper closed
$G$-invariant subschemes. 

For brevity, we introduce the notation
\[
\dgm X = D^-(\cg X)
\qquad\text{and}\qquad
\dgb X = \Db(\cg X)
\]
for the bounded-above and bounded derived categories of $\cg X$.  These
are equivalent to the full triangulated subcategories of $D^-(\qg X)$ and
$\Db(\qg X)$, respectively, consisting of objects with coherent
cohomology.  (See~\cite[Corollary~1]{bez:pc}.)  We also let
\[
\dgp X = \{ \cF \in D^+(\qg X) \mid \text{$\cF$ has coherent cohomology}
\}.
\]
$\dgp X$ is not, in general, equivalent to $D^+(\cg X)$, but it is the
right category to work in from the viewpoint of Grothendieck--Serre
duality.  (We omit the subscript $G$ for the corresponding nonequivariant derived categories.)  The cohomology sheaves of an object $\cF$ in one of these derived categories will be denoted $H^k(\cF)$.  In particular, this notation will \emph{not} be used for derived functors of the global-section functor $\Gamma$.

We write $\dgm X^{\le n}$ and $\dgb X^{\le n}$ for the full subcategories
consisting objects $\cF$ with $H^k(\cF) = 0$ for $k > n$.  $\dgp X^{\ge
n}$ and $\dgb X^{\ge n}$ are defined similarly, and we put $\dgb X^{[a,b]}
= \dgb X^{\ge a} \cap \dgb X^{\le b}$.  We also have truncation functors $\tau^{\le n}$ and $\tau^{\ge n}$ for each $n$, and the composition $\tau^{[a,b]} = \tau^{\ge a} \circ \tau^{\le b}$.

Let $Z \subset X$ be a closed subscheme.  The notation
\[
V \Subset_Z X
\]
will be used to mean that $V$ is an open subscheme of $X$ containing the
complement of $Z$, and with the property that $V \cap Z$ is a dense open
subscheme of $Z$.  (In particular, $V$ is a dense open subscheme of $X$.) 
We also use the notation
\[
V \Subset X
\]
as a synonym for $V \Subset_X X$, {\it i.e.}, to say that $V$ is a
dense open subscheme of $X$.

Given $\cF, \cG \in \cg X$, we write $\cHom(\cF,\cG)$ for their ``internal Hom,'' also an object of $\cg X$.  This $\cHom$ commutes with the forgetful functor to the nonequivariant category.  Note that $\Gamma(\cHom(\cF,\cG)) \not\simeq \Hom(\cF,\cG)$ in general, as the latter is the group of $G$-equivariant morphisms.

Next, we define $\gHom(\cF,\cG)$ to be the rule assigning to each open subscheme $V \subset X$ the group $\Hom(\cF|_V,\cG|_V)$.  This is not usually a sheaf, as it does not take values on non-$G$-invariant open subschemes.  We define $\gExt^r(\cF,\cG)$ for any $r\ge 0$ similarly.  We will not formally develop the theory of such objects; rather, the main use of this notation will be to enable us abbreviate certain kinds of statements and arguments.  For instance, we may say ``$\gHom(\cF|_U, \cG|_U) = 0$'' rather than ``for all $V \subset U$, $\Hom(\cF|_V,\cG|_V) = 0$,'' and we may state that there is an exact sequence
\[
\gHom(\cF,\cG'') \to \gExt^1(\cF,\cG') \to \gExt^1(\cF,\cG)
\]
to mean that for every open subscheme $V \subset X$, there is an exact sequence
\[
\Hom(\cF|_V,\cG''|_V) \to \Ext^1(\cF|_V,\cG'|_V) \to
\Ext^1(\cF|_V,\cG|_V).
\]

If $i: Z \hto X$ is a closed subscheme of $X$, the underlying topological space of $Z$ will be denoted $\uZ$.  An object $\cF$ of $\cg X$ or $\dg X$ is said to be \emph{supported on $Z$} if $\cF \simeq i_*\cF_1$ for some object $\cF_1$
of $\cg Z$ or $\dg Z$.  On the other hand, $\cF$ is \emph{supported on $\uZ$} simply if all nonzero stalks of $\cF$ are over points of $\uZ$.  We define two categories related to $\uZ$ as follows:
\begin{align*}
\czsupp &= \{ \cF \in \cg X \mid \text{$\cF$ is supported on $\uZ$} \}, \\
\dzsupp &= \{ \cF \in \dgb X \mid \text{$\cF$ is supported on $\uZ$} \}.
\end{align*}
Note that objects of $\dzsupp$ are by definition bounded.
Recall that for every object $\cF$ in $\czsupp$ or $\dzsupp$, there exists a closed subscheme structure $i_{Z'}: Z' \hto X$ and an object $\cF_1$ in $\cg{Z'}$ or $\dgb{Z'}$ such that $\cF \simeq i_{Z'*}\cF_1$.

For a closed subscheme $i: Z \hto X$, the notation $i^*$ will always denote the coherent pull-back functor $i^*: \cg X \to \cg Z$.  We denote by $i^\flat: \cg X \to \cg Z$ the
functor of ``sections supported on $Z$.''  This functor is right adjoint to
$i_*: \cg Z \to \cg X$, and satisfies $i_*i^\flat\cF \simeq \cHom(i_*\cO_Z,
\cF)$.  Note that $i_*i^\flat\cF$ is naturally a subsheaf of $\cF$.

Finally, we let $\Gamma_Z: \cg X \to \cg X$ be the functor of ``sections
supported on $\uZ$.''  $\Gamma_Z\cF$ is a subsheaf of $\cF$, and we have a
natural isomorphism
\[
\Gamma_Z\cF \simeq \lim_{\substack{\to \\ Z'}} i_{Z'*}i_{Z'}^\flat \cF
\]
where $i_{Z'}: Z' \hto X$ ranges over all closed subscheme structures on $\uZ$.

\subsection{Equivariant derived categories}
\label{subsect:equiv-der}

There are a number of foundational issues to be addressed in translating the theory of derived categories and functors of coherent sheaves from the nonequivariant setting (following~\cite{har}) to the equivariant one.  Many of these have been treated by Bezrukavnikov~\cite[Section~2]{bez:pc}.  Here, we briefly consider the issues that are most relevant to the present paper.  Consider first the derived functors
\[
Li^*: \dgm X \to \dgm Z \qquad\text{and}\qquad
\cRHom: (\dgm X)^{\mathrm{op}} \times \dgp X \to \dgp X.
\]
$Li^*$ may be computed by taking a locally free resolution, since $\cg X$ is assumed to have enough locally free objects.  Similarly, $\cRHom$ may be computed by taking a locally free resolution in the first variable.  We will also require the derived functor of $i^\flat$.  The quasicoherent category $\qg X$ has enough injectives, so there is no problem in constructing the functor
\[
Ri^\flat: D^+(\qg X) \to D^+(\qg Z),
\]
but it is not yet clear how to obtain from this a functor on $\dgp X$.  We will address this issue below.

\begin{lem}\label{lem:found}
Let $j: U \hto X$ be an open subscheme, and let $i: Z \hto X$ be a complementary closed subscheme.
\begin{enumerate}
\item The functor $Li^*$ is left-adjoint to $i_*$, and $Ri^!$ is right-adjoint to $i_*$.\label{it:f-adj}
\item The functors $Li^*$, $\cRHom$, and $Ri^\flat$ commute with the forgetful functor to the appropriate nonequivariant derived category.  In particular, $Ri^\flat$ restricts to a functor $\dgp X \to \dgp Z$. \label{it:f-forg}
\item For any $\cF \in \dgm X$ and $\cG \in \dgp X$, there is a long exact sequence\label{it:f-les}
\[
\to \lim_{\substack{\to \\ Z'}} \Hom(Li^*_{Z'}\cF, Ri^\flat_{Z'}\cG) \to
\Hom(\cF,\cG) \to \Hom(j^*\cF,j^*\cG) \to
\]
where the limit runs over all subscheme structures $i_{Z'}: Z' \hto X$ on $\uZ$.
\end{enumerate}
\end{lem}
\begin{proof}
\eqref{it:f-adj}
The adjointness properties for $Li^*$, $Ri^\flat$, and $i_*$ follow by standard arguments from the corresponding adjointness properties at the level of abelian categories.  (For now, the adjointness of $i_*$ and $Ri^\flat$ holds only in the quasicoherent setting.)

\eqref{it:f-forg}
Because both $Li^*$ and $\cRHom$ can be computed by locally free resolutions in both the equivariant and the nonequivariant cases, the fact that they commute with the forgetful functor $\Forg$ follows from the fact that $\Forg$ takes locally free sheaves in $\cg X$ to locally free sheaves in $\cn X$.

Before dealing with $Ri^\flat$, we first consider another approach to computing $\cRHom$, via certain resolutions in the second variable.  The subtlety here is that $\Forg$ does not, in general, take injective objects of $\qg X$ to injective objects of $\qn X$.  To get around this problem, we make use of the ``averaging functor'' $\Av = a_*pr^* : \qn X \to \qg X$, where $pr: G \times X \to X$ is projection on to the second factor, and $a: G \times X \to X$ is the action map.  $\Av$ is exact and right-adjoint to $\Forg$, and it takes injective objects to injective objects (see~\cite[Section~2]{bez:pc}).  Let us say that an object of $\qg X$ is \emph{$\Av$-injective} if it is isomorphic to $\Av(\cI)$ for some injective object $\cI \in \qn X$.  Then every object of $\qg X$ is a subsheaf of some $\Av$-injective sheaf, and every object in $D^+(\qg X)$ admits an $\Av$-injective resolution.  We claim that $\cRHom$ can be computed by $\Av$-injective resolutions in the second variable.  This claim would follow from the statement that if $\cI \in \qn X$ is injective, then $\cHom(\cdot, \Av(\cI))$ is an exact functor on $\qg X$.  The latter holds because
\[
\cHom(\cdot, \Av(I)) \simeq \Av(\cHom(\Forg(\cdot),\cI)),
\]
and because $\cHom(\cdot,\cI)$ is an exact functor on $\qn X$.  

Now, for a coherent sheaf $\cF \in \cg X$, consider the sheaf of abelian groups on $Z$ given by
\[
\cHom(i_*\cO_Z, \cF)|_Z.
\]
Here, ``$|_Z$'' denotes the exact functor of restriction to $Z$ in the usual sense for sheaves of abelian groups, \emph{not} the coherent pull-back functor.  Nevertheless, it easy to see that this sheaf can be naturally regarded as an object of $\cg Z$.  Indeed, because $i_*i^\flat\cF \simeq \cHom(i_*\cO_Z,cF)$, we have
\[
i^\flat\cF \simeq \cHom(i_*\cO_Z, \cF)|_Z.
\]
We can use this formula to better understand $Ri^\flat$.  For $\cF \in \dgp X$, $Ri^\flat\cF$ can be computed by applying $i^\flat \simeq \cHom(i_*\cO_Z, \cdot)|_Z$ to an $\Av$-injective resolution of $\cF$.  By the previous paragraph, this is equivalent to taking a chain complex representing $\cRHom(i_*\cO_Z,\cF)$ and then applying ``$|_Z$.''  Since $\cRHom$ and restriction to $Z$ both commute with $\Forg$, $Ri^\flat$ does as well.  In particular, when evaluated on objects of $\dgp X$, $Ri^\flat$ takes values in $\dgp Z$ ({\it i.e.}, it has coherent cohomology), because the analogous statement is true for the corresponding nonequivariant functor.

\eqref{it:f-les}
We follow the proof given in~\cite[Proposition~2]{bez:pc}.  In $D^+(\qg X)$, there is a well-known distinguished triangle $R\Gamma_Z\cG \to \cG \to Rj_*j^*\cG \to$, from which we obtain a long exact sequence
\[
\to \Hom(\cF,R\Gamma_Z\cG) \to \Hom(\cF,\cG) \to \Hom(\cF, Rj_*j^*\cG) \to
\]
By~\cite[Lemma~3(b)]{bez:pc}, we have
\[
\Hom(\cF, R\Gamma_Z\cG) \simeq \lim_{\substack{\to \\ Z'}} \Hom(\cF, i_{Z'*}Ri^\flat_{Z'}\cG) \simeq \lim_{\substack{\to \\ Z'}} \Hom(Li^*_{Z'}\cF, Ri^\flat_{Z'}\cG),
\]
where $i_{Z'}: Z' \hto X$ runs over all subscheme structures on $\uZ$, 
and on the other hand, we have $\Hom(\cF, Rj_*j^*\cG) \simeq \Hom(j^*\cF, j^*\cG)$. 
\end{proof}

According to~\cite[Proposition~1]{bez:pc}, the scheme $X$ admits an equivariant dualizing complex, {\it i.e.}, an object $\omega_X \in \dgb X$ such that the functor $\D = \cRHom(\cdot, \omega_X) : \dgb X \to \dgb X$ has the property that $\D \circ \D \simeq \id$.  The functor $\D$ can be evaluated on objects of $\dgm X$; it then takes values in $\dgp X$.  

One can also evaluate $\D$ on objects of $\dgp X$: remarkably, its values, {\it a priori} belonging to the unbounded derived category $D(\qg X)$, actually lie in $\dgm X$.  This fact follows from the corresponding nonequivariant statement, since we know that $\cRHom$ commutes with $\Forg$, and $\Forg(\omega_X)$ is a nonequivariant dualizing complex~\cite[Lemma~4]{bez:pc}.  The nonequivariant version holds because $\Forg(\omega_X)$ is quasi-isomorphic to a bounded complex of injective objects of $\qn X$~\cite[II.7.20]{har}.  (Note that the latter argument cannot simply be repeated in the equivariant case: $\omega_X$ need not be quasi-isomorphic to a bounded complex of injective objects of $\qg X$.)

The functor $\D$ gives antiequivalences in both directions:
\[
\D: \dgm X \to \dgp X, \qquad \D: \dgp X \to \dgm X.
\]
In general, dualizing complexes are not uniquely determined, but if we are given a dualizing complex $\omega_X$ on $X$, we can construct from it a dualizing complex on any open subscheme $j: U \hto X$ by the formula
\[
\omega_U = \omega_X|_U.
\]
On the other hand, for a closed subscheme  $i: Z \hto X$, the object 
\[
\omega_Z = Ri^\flat \omega_X
\]
is an equivariant dualizing complex.  (Again, by Lemma~\ref{lem:found} and~\cite[Lemma~4]{bez:pc}, this statement follows from the corresponding nonequivariant one, proved in~\cite[V.2.4]{har}.)  In particular, $\omega_Z$ is necessarily in $\dgb Z$.
From Section~\ref{sect:duality} on, we will have a fixed dualizing complex $\omega_X$ in mind for $X$, and we adopt the convention that on any locally closed subscheme of $X$, $\D$ is to be computed with the dualizing complex obtained from $\omega_X$ by the above formulas.

\begin{lem}\label{lem:found-dual}
The functor $\D$ commutes with the forgetful functor to the nonequivariant derived category.  For a closed subscheme $i: Z \hto X$, we have $\D \circ Li^* \simeq Ri^\flat \circ \D$.
\end{lem}
\begin{proof}
The first part of the lemma is immediate from Lemma~\ref{lem:found}.  Next, recall that for $\cF$ and $\cG$ in $\cg X$, we have a natural isomorphism
\[
\cHom(i^*\cF, i^\flat\cG) \simeq i^\flat \cHom(\cF,\cG).
\]
By the usual general methods (see~\cite[Proposition~III.6.9(b)]{har} for the nonequivariant case), we may deduce that
\[
\cRHom(Li^*\cF, Ri^\flat\cG) \simeq Ri^\flat \cRHom(\cF,\cG)
\]
for $\cF \in \dgm X$ and $\cG \in \dgp X$.  Now, take $\cG = \omega_X$.  In this case, we know that $Ri^\flat\omega_X \simeq \omega_Z$.  The isomorphism above then becomes $\D \circ Li^* \simeq Ri^\flat \circ \D$.
\end{proof}

\subsection{Coherent sheaves on open subschemes}

We now prove several useful lemmas relating coherent sheaves on $X$ to those on an open subscheme.

\begin{lem}\label{lem:subcplx-extend}
Let $j: U \hto X$ be an open subscheme, and let $\cF \in \dgb X^{[a,b]}$.
Given a distinguished triangle
\begin{equation}\label{eqn:cplx-dtU}
\cF'_U \to \cF|_U \to \cF''_U \to
\end{equation}
in $\dgb U$, with $\cF'_U \in \dgb U^{[a,b]}$, there exist objects $\cF'
\in \dgb X^{[a,b]}$, $\cF'' \in \dgb X$ and a distinguished triangle
\[
\cF' \to \cF \to \cF'' \to
\]
in $\dgb X$ whose image under $j^*$ is isomorphic to the distinguished
triangle given in~\eqref{eqn:cplx-dtU}.
\end{lem}
\begin{proof}
Fix a bounded complex $\cG^\bullet$ of objects of $\cg X$ that
represents $\cF$ and such that $\cG^i = 0$ for $i < a$ or $i > b$.  Then,
there exists a complex $\cG'_U{}^\bullet$ of objects of $\cg U$ and a
morphism of
complexes $f: \cG'_U{}^\bullet \to \cG^\bullet|_U$ representing the
morphism $\cF'_U \to \cF|_U$.  We may again assume that $\cG'_U{}^i = 0$
for $i < a$ or $i > b$.  Now, $f$ induces a morphism of complexes of
quasicoherent sheaves $j_*f: j_*\cG'_U{}^\bullet \to j_*j^*\cG^\bullet$. 
There is also a natural morphism $\cG^\bullet \to j_*j^*\cG^\bullet$. 
Consider the diagram
\[
\xymatrix@=10pt{
& \cG^\bullet \ar[d] \\
j_*\cG'_U{}^\bullet \ar[r]^{j_*f} & j_*j^*\cG^\bullet}
\]
There
certainly exist coherent subcomplexes of $j_*\cG'_U{}^\bullet$ whose
restriction to $U$ is $\cG'_U{}^\bullet$.  Choose one such complex, and denote it $\cH_1^\bullet$.  Now, form the pullback 
\[
\xymatrix@=10pt{
\cG'{}^\bullet \ar@{.>}[r] \ar@{.>}[d] & \cG^\bullet \ar[d] \\
\cH_1^\bullet \ar[r] & j_*j^*\cG^\bullet}
\]
When we apply the exact functor $j^*$, the right-hand vertical arrow
becomes an isomorphism, and after passing to $\dgb U$, the arrow along the
bottom becomes $\cF'_U \to \cF|_U$.  Therefore, the left-hand vertical
arrow also becomes an isomorphism, and $j^*\cG'{}^\bullet \to
j^*\cG^\bullet$ is isomorphic in $\dgb U$ to $\cF'_U \to \cF|_U$.  Let
$\cF'$ be $\cG'{}^\bullet$ regarded as an object of $\dgb X$, and let
$\cF''$ be the cone of $\cF' \to \cF$.  Then $\cF' \to \cF \to \cF'' \to$
is the distinguished triangle we seek.
\end{proof}

When $a = b = 0$, the previous lemma yields the following abelian-category statement.

\begin{lem}\label{lem:subsheaf-extend}
Let $U \subset X$ be an open subscheme.  Let $\cF \in \cg X$.  For any
short exact sequence $0 \to \cF'_U \to \cF|_U \to \cF'_U \to 0$ of sheaves on $U$,
there exists a short exact sequence $0 \to \cF' \to \cF \to \cF'' \to
0$ of sheaves on $X$ whose image under $j^*$ is the given short exact
sequence.\qed
\end{lem}

\begin{lem}\label{lem:open0}
Let $\cC$ be a full subcategory of $\cg X$, closed under subobjects.  Let
$i: Z \hto X$ be a closed subscheme, and let $j: U \hto X$ be the
complementary open subscheme.  If $\cG \in \cg X$ is such that $\Hom(\cF,
\cG/\Gamma_Z\cG) = 0$ for all $\cF \in \cC$, then $\Hom(j^*\cF, j^*\cG) = 0$
for all $\cF \in \cC$.
\end{lem}
\begin{proof}
Suppose $\Hom(j^*\cF,j^*\cG) \ne 0$ for some $\cF \in \cC$.  Then there is
a nonzero morphism of quasicoherent sheaves $f: j_*j^*\cF \to j_*j^*\cG$
such that $j^*f$ is also nonzero.  Now, consider the following diagram
\[
\xymatrix@=10pt{
& \cF \ar[r]^p & j_*j^*\cF \ar[d]^f \\
\Gamma_Z\cG \ar[r] & \cG \ar[r]^q & j_*j^*\cG}
\]
The bottom row is left exact, and in particular, the image of $q$ is isomorphic to $\cG/\Gamma_Z\cG$.  Now, let $\cH
= \im(f \circ p)$.  This is a coherent subsheaf of $j_*j^*\cG$.  Moreover,
$\cH \cap \im q$ must be nonzero, because when we apply the exact functor
$j^*$ to this diagram, we find that $j^*\cH \simeq \im(j^*f)$ is a nonzero
subsheaf of $j^*(j_*j^*\cG) \simeq j^*\cG$.  Let $\cF' \subset \cF$ be the
preimage via $f \circ p$ of $\cH \cap \im q$.  This is a nonzero subsheaf
of $\cF$, and so a member of $\cC$.  The restriction of $f \circ p$ to
$\cF'$ is a nonzero morphism $\cF' \to \im q \simeq \cG/\Gamma_Z\cG$, a
contradiction.
\end{proof}

\subsection{Abelian categories}

We conclude this section with two useful propositions about
abelian categories.  Recall that an abelian category $\cA$ is
said to be \emph{noetherian}
if for every object $A \in \cA$, the ascending chain condition holds for
subobjects of $A$.

\begin{prop}\label{prop:rt-adj}
Let $\cA$ be a noetherian abelian category, and let $\cA'$ be a full subcategory closed under extensions and quotients.  Then the inclusion functor $\iota: \cA' \to \cA$ admits a right adjoint $\rho: \cA \to \cA'$.
\end{prop}
\begin{proof}
Given $A \in \cA$, consider the set
$S$ of all subobjects of $A$ belonging to $\cA'$.  $S$ is partially
ordered by inclusion, and totally ordered subsets of $S$ have a maximum
element, by the ascending chain condition.  By Zorn's lemma, $S$ contains
at least one maximal element.

Now, suppose $S$ contained two distinct maximal elements, $B_1$ and $B_2$.
Form the cartesian and cocartesian square
\[
\xymatrix@=10pt{
B_1 \cap B_2 \ar[r]\ar[d] & B_1 \ar[d] \\
B_2 \ar[r] & B_1 + B_2}
\]
Now, $B_1/(B_1 \cap B_2) \simeq (B_1+B_2)/B_1$ is in
$\cA'$, since $\cA'$ is closed under quotients.
But $B_1 + B_2$ is an extension of $(B_1+B_2)/B_1$ by $B_1$,
and since $\cA'$ is closed under extensions, $B_1 + B_2 \in
\cA'$.  This contradicts the maximality of $B_1$ and $B_2$ in $S$,
so $S$ cannot contain two distinct maximal elements.

Therefore, $S$ contains a unique maximal element $A'$.  Now, let $B \in \cA'$.  For any morphism $f: B \to A$, the image of $f$ is a subobject of $A$ that is in $\cA'$, since $\cA'$ is closed under quotients.  Therefore, the image of $f$ is actually a subobject of $A'$, and $f$ factors through the inclusion $A' \hto A$.  We have shown that
\[
\Hom(B,A) \simeq \Hom(B,A'),
\]
so the functor $\rho: A \mapsto A'$ is the desired right adjoint.
\end{proof}

Next, recall that the (\emph{right}) \emph{orthogonal complement} of a
subcategory $\cA' \subset \cA$ is the full subcategory defined by
\[
(\cA')^\perp = \{ B \mid \text{$\Hom(A,B) = 0$ for all $A \in \cA'$} \}.
\]

\begin{prop}\label{prop:noeth-orth}
Let $\cA$ be a noetherian abelian category, and let $\cA'$ be a full
subcategory closed under extensions and quotients.  If $\cA''$ is full
subcategory of $\cA$ contained in $(\cA')^\perp$, then the following
conditions are equivalent:
\begin{enumerate}
\item $\cA'' = (\cA')^\perp$.\label{it:perp}
\item For every object $A \in \cA$, there is a short exact sequence
\begin{equation}\label{eqn:ab-ses}
0 \to A' \to A \to A'' \to 0
\end{equation}
with $A' \in \cA'$ and $A'' \in \cA''$.\label{it:ses}
\end{enumerate}
When these conditions hold, the inclusion functor $\cA'' \to \cA$ admits a
left adjoint $\lambda: \cA \to \cA''$.  Moreover, for every $A \in \cA$,
there is a natural short exact sequence
\[
0 \to \rho(A) \to A \to \lambda(A) \to 0,
\]
where $\rho$ is the right adjoint to the inclusion $\cA' \to \cA$, and in
addition every short exact sequence of the form~\eqref{eqn:ab-ses} is
canonically isomorphic to this one.
\end{prop}
\begin{proof}
For any $A \in \cA$, we claim that $A/\rho(A) \in (\cA')^\perp$.  Indeed,
if there were a nonzero morphism $A_1 \to A/\rho(A)$ with $A_1 \in \cA'$,
then
its image $\im f$ would be a subobject of $A/\rho(A)$ in $\cA'$.  Since
$\cA'$ is stable under extensions, the preimage in $A$ of $\im f$ (an
extension of $\im f$ by $\rho(A)$) would be a subobject of $A$ in $\cA'$
containing $\rho(A)$ as a proper subobject.  But that contradicts the
characterization of $\rho(A)$ as the maximal subobject of $A$ in $\cA'$.

It is now clear that condition~\eqref{it:perp} implies
condition~\eqref{it:ses}: for any $A \in \cA$, the short exact sequence
\[
0 \to \rho(A) \to A \to A/\rho(A) \to 0
\]
is of the form~\eqref{eqn:ab-ses}.  

Conversely, suppose condition~\eqref{it:ses} holds.  Given $A \in
(\cA')^\perp$, we must show that $A \in \cA''$.  Form a short exact
sequence
\[
0 \to A' \to A \to A'' \to 0
\]
as in~\eqref{eqn:ab-ses}.  Now, the morphism $A' \to A$ must be $0$, since
$A \in (\cA')^\perp$, so in fact $A \simeq A'' \in \cA''$, as desired.

We now prove the last part of the proposition.  Given $A \in \cA$,
form a short exact sequence as in~\eqref{eqn:ab-ses}.  Now, for any $B \in
\cA''$, we have an exact sequence
\[
0 \to \Hom(A'',B) \to \Hom(A,B) \to \Hom(A',B).
\]
The last term vanishes because $\cA'' \subset (\cA')^\perp$, so if we put
$\lambda(A) = A''$, we have an isomorphism $\Hom(\lambda(A),B)
\overset{\sim}{\to} \Hom(A,B)$.  Now, we know that $A' \to A$ must factor
through $\rho(A) \to A$, so $\rho(A)/A'$ can be identified with a subobject
of $A''$.  But $\rho(A)/A' \in \cA'$, and since $A'' \in (\cA')^\perp$, we
must have $\rho(A) \simeq A'$, and hence that $\lambda(A) \simeq
A/\rho(A)$.  This proves the functoriality of $\lambda$ (which is now
clearly left adjoint to $\cA'' \to \cA$) and the uniqueness of the short
exact sequence~\eqref{eqn:ab-ses}.
\end{proof}

\section{$s$-structures on Schemes}
\label{sect:s-struc}

In this section, we define and establish basic properties of
$s$-structures.  The definition is quite lengthy and is given in several
steps, partly because it will be convenient in the sequel to be able to
discuss situations in which only part of the full definition is satisfied.

\begin{defn}\label{defn:pre-s}
A \emph{pre-$s$-structure} on $X$ is a collection of full subcategories
\[
(\{\cgl Xw\}, \{\cgg Xw\})_{w \in \Z}
\]
of $\cg X$ such that:
\begin{enumerate}
\renewcommand{\labelenumi}{(S\arabic{enumi})}
\item Each $\cgl Xw$ is a Serre subcategory, and each $\cgg Xw$ is closed
under subobjects and extensions.\label{ax:serre}
\item $\cgl Xw \subset \cgl X{w+1}$ and $\cgg Xw \supset \cgg
X{w+1}$.\label{ax:inc}
\item If $\cF \in \cgl Xw$ and $\cG \in \cgg X{w+1}$, then $\gHom(\cF,\cG) =
0$.\label{ax:hom}
\item For any $\cF \in \cg X$ and any $w \in \Z$, there is a short exact
sequence
\[
0 \to \cF' \to \cF \to \cF'' \to 0
\]
with $\cF' \in \cgl Xw$ and $\cF'' \in \cgg X{w+1}$.\label{ax:ses}
\item For any $\cF \in \cg X$, there are integers $w,v \in \Z$ such that
$\cF \in \cgg Xw$ and $\cF \in \cgl Xv$.\label{ax:bdd}
\item If $\cF \in \cgl Xw$ and $\cG \in \cgl Xv$, then $\cF \otimes \cG \in
\cgl X{w+v}$.\label{ax:tensorl}
\end{enumerate}
\end{defn}

By Propositions~\ref{prop:rt-adj} and~\ref{prop:noeth-orth}, the inclusion
functors $\cgl Xw \to \cg X$ and $\cgg Xw \to \cg X$ admit right and left
adjoints, respectively.  We denote these by $\sigma_{\le w}: \cg X \to
\cgl Xw$ and $\sigma_{\ge w}: \cg X \to \cgg Xw$.  Thus, for any $w$,
there is a short exact sequence
\[
0 \to \sigma_{\le w}\cF \to \cF \to \sigma_{\ge w+1}\cF \to 0.
\]
From the proof of Proposition~\ref{prop:rt-adj}, we see that we may regard
$\sigma_{\le w}\cF$ as being the largest subsheaf of $\cF$ in $\cgl Xw$.

\begin{rmk}\label{rmk:homl}
Suppose $(\{\cgl Xw\}, \{\cgg Xw\})_{w \in \Z}$ is a collection of full
subcategories of $\cg X$ satisfiying
axioms~\sref{ax:serre}--\sref{ax:ses} above.  Then
axiom~\sref{ax:tensorl} is equivalent to the following useful condition:
\begin{enumerate}
\item[(S\ref{ax:tensorl}$'$)] If $\cF \in \cgl Xw$ and $\cG \in \cgg Xv$, then $\cHom(\cF,\cG) \in
\cgg X{v-w}$.
\end{enumerate}
Indeed, suppose axiom~\sref{ax:tensorl} holds, and suppose $\cF
\in \cgl Xw$ and $\cG \in \cgg Xv$.  By Proposition~\ref{prop:noeth-orth},
we know that $\cgg X{v-w} = \cgl X{v-w-1}^\perp$, so to show that
$\cHom(\cF,\cG) \in \cgg X{v-w}$ is equivalent to showing that
\[
\Hom(\cH, \cHom(\cF,\cG)) = 0\qquad\text{for all $\cH \in \cgl X{v-w-1}$.}
\]
But $\Hom(\cH, \cHom(\cF,\cG)) \simeq \Hom(\cH \otimes \cF, \cG) = 0$
because $\cH \otimes \cF \in \cgl X{v-1}$.  The same arugment shows that condition~(S\ref{ax:tensorl}$'$) implies axiom~\sref{ax:tensorl} as well.
\end{rmk}

\begin{defn}
Suppose $X$ has a pre-$s$-structure, and let $\cF \in \cg X$.  If there is an integer $w$ such that $\cF \in \cgl Zw$ and $\sigma_{\le w-1}\cF = 0$, then $\cF$ is said to be \emph{pure of step $w$}, and we write $w = \step \cF$.
\end{defn}

\begin{defn}
Let $X$ and $Y$ be two schemes with pre-$s$-structures.  A functor $F: \cg
X \to \cg Y$ is said to be \emph{left} (resp.~\emph{right})
\emph{$s$-exact} if $F(\cgg Xw) \subset \cgg Yw$ (resp.~$F(\cgl Xw)
\subset \cgl Yw$).  It is \emph{$s$-exact} if it is both left $s$-exact
and right $s$-exact.
\end{defn}

Next, we show that a pre-$s$-structure gives rise to canonical
pre-$s$-structures on any open or closed (and hence any locally closed)
subscheme.

\begin{prop}\label{prop:res-open}
Let $j: U \hto X$ be an open subscheme.  Given a pre-$s$-structure on $X$,
there is a unique pre-$s$-structure on $U$ such that $j^*$ is $s$-exact. 
It is
given by
\begin{equation}\label{eqn:res-open}
\begin{aligned}
\cgl Uw &= \{ \cF \mid \text{there exists an $\cF_1 \in
\cgl Xw$ such that $j^*\cF_1 \simeq \cF$} \}, \\
\cgg Uw &= \{ \cF \mid \text{there exists an $\cF_1 \in
\cgl Xw$ such that $j^*\cF_1 \simeq \cF$} \}.
\end{aligned}
\end{equation}
\end{prop}
\begin{proof}
Once we show that these categories constitute a pre-$s$-structure, the uniqueness statement is obvious.  It is clear that axioms~\sref{ax:inc} and~\sref{ax:bdd} hold, and axiom~\sref{ax:hom}
holds by the definition of $\gHom$.  Given any $\cF \in \cg U$, let $\cF_1$
be any extension of it to a coherent sheaf on $X$, and form the exact
sequence
\[
0 \to \sigma_{\le w}\cF_1 \to \cF_1 \to \sigma_{\ge w+1}\cF_1 \to 0.
\]
Applying $j^*$ to this sequence gives a short exact sequence in $\cg U$
showing that axiom~\sref{ax:ses} holds.  Axiom~\sref{ax:tensorl} follows from the
formula $j^*\cF \otimes j^*\cG \simeq j^*(\cF \otimes \cG)$.

It remains to prove axiom~\sref{ax:serre}.  Let $0 \to \cF' \to \cF \to \cF'' \to 0$ be a short exact sequence in $\cg
U$.  If $\cF \in \cgl Uw$, let $\cF_1 \in \cgl Xw$ be a sheaf such that
$j^*\cF_1 \simeq \cF$.  Using Lemma~\ref{lem:subsheaf-extend} and the fact
that $\cgl Xw$ is stable under subobjects and quotients, we find that
$\cF'$ and $\cF''$ lie in $\cgl Uw$.  Thus, $\cgl Uw$ is stable under
subobjects and quotients.  On the other hand, suppose $\cF'$ and $\cF''$ both belong to $\cgl Uw$.  If $\cF \notin \cgl Uw$, then from axiom~\sref{ax:ses}, we obtain a nonzero morphism $\cF \to \cG$ with $\cG \in \cgg U{w+1}$, but this is a contradiction, since $\Hom(\cF',\cG) = \Hom(\cF'',\cG) = 0$ by axiom~\sref{ax:hom}.  Thus $\cF \in \cgl Uw$, and $\cgl Uw$ is stable under extensions.  The same arguments show that $\cgg Uw$ is stable
under subobjects and extensions.  
\end{proof}

\begin{prop}\label{prop:res-closed}
Let $i: Z \hto X$ be a closed $G$-invariant subscheme of
$X$.  Given a pre-$s$-structure on $X$, there is a pre-$s$-structure on
$Z$ given by
\begin{equation}\label{eqn:res-closed}
\begin{aligned}
\cgl Zw &= \{ \cF \mid i_*\cF \in \cgl Xw \}, \\
\cgg Zw &= \{ \cF \mid i_*\cF \in \cgg Xw \}.
\end{aligned}
\end{equation}
This is the unique pre-$s$-structure on $Z$
such that the functor $i^*$ is right $s$-exact and the functor $i^\flat$ is
left $s$-exact.  Moreover, the functor $i_*: \cg Z \to \cg X$ is
$s$-exact.
\end{prop}
\begin{proof}
For any two sheaves $\cF, \cG \in \cg Z$, we have
\[
\Hom(\cF,\cG) \simeq \Hom(i^*i_*\cF, \cG) \simeq \Hom(i_*\cF, i_*\cG).
\]
We also have $i_*\cHom(\cF,\cG) \simeq \cHom(i_*\cF,i_*\cG)$ and $i_*(\cF
\otimes \cG) \simeq i_*\cF \otimes i_*\cG$.  In view of these observations
and the fact that $i_*$ is an exact functor, it is straightforward to
verify all the axioms for a pre-$s$-structure. We omit the details.

If $\cF \in \cgl Xw$, then $i_*i^*\cF$ is a quotient of $\cF$, and hence in
$\cgl Xw$. Therefore, $i^*\cF \in \cgl Zw$, so $i^*$ is right $s$-exact.  A
similar argument shows that $i^\flat$ is left $s$-exact.  The uniqueness of
this $s$-structure follows from the fact that for any $\cF \in \cg Z$,
we have $i^*i_*\cF \simeq i^\flat i_*\cF \simeq \cF$.  The $s$-exactness of
$i_*$ is obvious.
\end{proof}

Let $\kappa: Y \hto X$ be a closed subscheme, and assume $Y$ is endowed
with a pre-$s$-structure.  We define a full subcategory of $\cg X$ by
\begin{equation}\label{eqn:cloc-defn}
\cloc XYw = \{ \cF \in \cg X \mid \kappa^*\cF \in \cgl Yw \}.
\end{equation}
Later, in Proposition~\ref{prop:cloc-serre}, we will see that this is a
Serre subcategory of $\cg X$.

\begin{defn}\label{defn:s}
Suppose $X$ is endowed with a pre-$s$-structure, and let
\[
\tcgg Xw = \{ \cF \mid \text{for some $V \Subset X$, $\cF|_V \in \cgg Vw$}
\}.
\]
The pre-$s$-structure is an \emph{almost $s$-structure} if for every closed
subscheme $\kappa: Y \hto X$, the induced pre-$s$-structure on $Y$
satisfies the following additional axioms:
\begin{enumerate}
\renewcommand{\labelenumi}{(S\arabic{enumi})}
\setcounter{enumi}{6}
\item $\tcgg Yw$ is a Serre subcategory of $\cg Y$.\label{ax:serrer}
\item If $\cF \in \tcgg Yw$ and $\cG \in \tcgg Yv$, then $\cF \otimes \cG
\in \tcgg Y{w+v}$.\label{ax:tensorr}
\end{enumerate}
Finally, an almost $s$-structure is an \emph{$s$-structure} if the following condition also holds for every closed subscheme $\kappa: Y \hto X$:
\begin{enumerate}
\renewcommand{\labelenumi}{(S\arabic{enumi})}
\setcounter{enumi}{8}
\item If $\cF \in \cloc XYw$, then for any $r \ge 0$, there is an open
subscheme $V \Subset_Y X$ such that $\gExt^r(\cF|_V, \kappa_*\cG|_V) = 0$
for all $\cG \in \tcgg Y{w+1}$.\label{ax:ext2}
\end{enumerate}
\end{defn}

It may seem at first glance that the last axiom would be quite
onerous to check in specific examples.  Fortunately, if $G$ acts on $X$
with finitely many orbits, this axiom follows from the others: see
Theorem~\ref{thm:adhesive}.

It is easy to see that in any pre-$s$-structure, $\tcgg Yw$ is necessarily
closed under subobjects and extensions, since $\cgg Vw$ is for every open
subscheme $V \subset Y$.  Thus, axiom~\sref{ax:serrer} is equivalent to
simply requiring that each $\tcgg Yw$ be closed under quotients.

\begin{lem}\label{lem:res-open-tcgg}
Let $j: U \hto X$ be an open subscheme, and suppose $X$ has a
pre-$s$-structure. Then
\begin{equation}\label{eqn:res-open-tcgg}
\tcgg Uw = \{ \cF \mid \text{there exists an
$\cF_1 \in \tcgg Xw$ such that $j^*\cF_1 \simeq \cF$} \}.
\end{equation}
Moreover, if $U$ is dense in $X$, then for all $\cF_1 \in \cg X$, we have $\cF_1 \in \tcgg Xw$ if and only if $j^*\cF_1 \in \tcgg Uw$.
\end{lem}
\begin{proof}
If $\cF_1 \in \tcgg Xw$, let $V \Subset X$ be such
that $\cF_1|_V \in \cgg Vw$.  Then $V \cap U \Subset U$, and clearly
$\cF_1|_{V \cap U} \simeq (j^*\cF_1)|_{V \cap U} \in \cgg
{V \cap U}w$, so we see that $j^*\cF_1 \in \tcgg Uw$.

Next, given $\cF \in \tcgg Uw$, let $V_1 \Subset U$ be such that
$\cF|_{V_1} \in \cgg{V_1}w$.  Let $\cF_1$ be any
extension of $\cF$ to a coherent sheaf on $X$ supported on $\overline U$. 
Let $V = V_1 \cup (X
\ssm \overline U)$.  Then we have $\cF_1|_V \in \cgg Vw$, as desired, so $\tcgg Uw$ is described by~\eqref{eqn:res-open-tcgg}.

Finally, in the case where $U$ is dense, the open subscheme $V$ of the previous paragraph coincides with $V_1$ and is dense in $X$, so the fact that $\cF_1|_V \in \cgg Vw$ means that $\cF_1 \in \tcgg Xw$.
\end{proof}

\begin{lem}\label{lem:der-hom}
Let $X$ be a scheme with an $s$-structure, and let $i: Z \hto X$ be a closed subscheme.  Let $\cF \in \dgb X$ be such that for
all $k$, there is an open subscheme $V_k \Subset_Z X$ such that
$H^k(\cF)|_{V_k} \in \cloc {V_k}{V_k \cap Z}w$.  Then, for any $r \in \Z$, there is a
open subscheme $V \Subset_Z X$ with the property that $\gHom(\cF|_V,
i_*\cG[r]|_V) = 0$ for all $\cG \in \tcgg Z{w+1}$.
\end{lem}
\begin{proof}
Suppose $\cF \in \dgb X^{[a,b]}$.  We proceed by induction on $b - a$.  If $b - a = 0$, then $\cF$ is simply a shift of an sheaf $\cF_1 \in \cg X$ with the property that $\cF_1|_{V_a} \in \cloc{V_a}{V_a \cap Z}w$.  The desired $V$ is obtained by invoking axiom~\sref{ax:ext2} on $V_a$.

Now, suppose $b > a$, and consider the distinguished triangle
\[
\tau^{\le a}\cF \to \cF \to \tau^{\ge a+1}\cF \to.
\]
This gives rise to an exact sequence
\[
\gHom(\tau^{\ge a+1}\cF,i_*\cG[r]) \to \gHom(\cF,i_*\cG[r]) \to \gHom(\tau^{\le a}\cF,i_*\cG[r]).
\]
By the inductive assumption, the first and last terms vanish for all $\cG \in \tcgg Z{w+1}$ after restriction to suitable dense open subschemes $V' \Subset_Z X$ and $V'' \Subset_Z X$, respectively.  Therefore, the middle term vanishes upon restriction to $V = V' \cap V''$.
\end{proof}

\begin{lem}\label{lem:li}
Let $X$ be a scheme with an $s$-structure, and let $i: Z \hto X$ and $\kappa: Y \hto Z$ be closed
subschemes.  For any sheaf $\cF \in \cloc XYw$ and any $r \ge 0$,
there is an open subscheme $V \Subset_Y Z$ such that $H^{-r}(Li^*\cF)|_V
\in \cloc V{V \cap Y}w$.
\end{lem}
\begin{proof}
We prove the lemma by induction on $r$.  For $r = 0$, the lemma is trivial: it is immediate from the definitions that $i^*\cF \in \cloc ZYw$.  Now, suppose $r > 0$.  For any $\cG \in \tcgg Y{w+1}$, we have $\gHom(Li^*\cF, \kappa_*\cG[r]) \simeq \gHom(\cF, i_*\kappa_*\cG[r]) \simeq \gExt^r(\cF, i_*\kappa_*\cG)$.  By axiom~\sref{ax:ext2}, there is an open subscheme $V' \Subset_Y X$ such that $\gExt^r(\cF|_{V'}, i_*\kappa_*\cG|_{V'}) = 0$ for all $\cG \in \tcgg Y{w+1}$.  Let $V'_1 = V' \cap Z$; then $\gHom(Li^*\cF|_{V'_1}, \kappa_*\cG[r]|_{V'_1}) = 0$ for all $\cG \in \tcgg Y{w+1}$.  In addition, we also have $\gHom(Li^*\cF, \kappa_*\cG[r]) \simeq \gHom(\tau^{\ge -r}Li^*\cF, \kappa_*\cG[r])$, so from the distinguished triangle
\[
\tau^{[-r,-r]}Li^*\cF \to \tau^{\ge -r}Li^*\cF \to \tau^{\ge -r+1}Li^*\cF \to
\]
we obtain the exact sequence
\begin{multline*}
\cdots \to \gHom(\tau^{\ge -r}Li^*\cF, \kappa_*\cG[r]) \to
\gHom(\tau^{[-r,-r]}Li^*\cF, \kappa_*\cG[r]) \to \\
\gHom(\tau^{\ge -r+1}Li^*\cF[-1], \kappa_*\cG[r]) \to \cdots.
\end{multline*}
Note that $\tau^{[-r,-r]}Li^*\cF \simeq H^{-r}(Li^*\cF)[r]$.  The sequence above can be rewritten as
\[
\gHom(Li^*\cF, \kappa_*\cG[r]) \to \gHom(H^{-r}(Li^*\cF), \kappa_*\cG) \to \gHom(\tau^{\ge -(r-1)}Li^*\cF, \kappa_*\cG[r+1]).
\]
The first term vanishes when we restrict to $V'_1$.  By the inductive
assumption, all cohomology sheaves of $\tau^{\ge -(r-1)}Li^*\cF$ satisfy the hypothesis of Lemma~\ref{lem:der-hom}, so invoking that lemma, we see that the last term above vanishes upon
restriction to some dense open subscheme $V'' \Subset_Y Z$.  Then, clearly,
the middle term vanishes upon restriction to $V = V'_1 \cap V''$.  Let $\tilde\kappa : V \cap Y \hto V$ denote the inclusion map.  The vanishing of the middle term above implies, in particular, that
\[
\Hom(H^{-r}(Li^*\cF)|_V, \tilde\kappa_*\cG_1) = \Hom(\tilde\kappa^* H^{-r}(Li^*\cF)|_V, \cG_1) = 0
\]
for all $\cG_1 \in \cgg {V \cap Y}{w+1}$, from which it follows in turn that $\tilde\kappa^*H^{-r}(Li^*\cF)|_V \in
\cgl {V \cap Y}w$.  Thus, $H^{-r}(Li^*\cF)|_V \in \cloc{V}{V \cap Y}w$, as desired.
\end{proof}

\begin{lem}
Let $X$ be a scheme with an $s$-structure.  The induced pre-$s$-structure
on any open or closed subscheme is an $s$-structure.
\end{lem}
\begin{proof}
Let $U \subset X$ be an open subscheme, and let $Y \subset U$ be a closed
subscheme.  It is straightforward to deduce the conditions in axioms~\sref{ax:serrer}--\sref{ax:ext2} from the corresponding facts for the closed subscheme $\overline Y \subset X$.  We omit the details.

Now, let $i: Z \hto X$ be a closed subscheme. 
Axioms~\sref{ax:serrer} and~\sref{ax:tensorr} hold for $Z$ automatically, since
any closed subscheme of $Z$ is also a closed subscheme of $X$.  

We now treat axiom~\sref{ax:ext2}.  Let $\kappa: Y \hto Z$ be a closed subscheme, and let $\cF \in \cloc ZYw$.  Fix an integer $r \ge 0$.  By the same axiom for $X$, we know that there exists an open subscheme $V_0 \Subset_Y X$ such that for all $\cG \in \tcgg Y{w+1}$, $\gExt^r(i_*\cF|_{V_0}, i_*\kappa_*\cG|_{V_0}) = 0$.  We assume inductively that~\sref{ax:ext2} is already known for $\gExt^k$-groups with $k < r$.  For each $k$ with $1 \le k \le r-1$, let us invoke Lemma~\ref{lem:li} to find an open subset $U_k \Subset_Y Z$ such that $H^{-k}(Li^*i_*\cF)|_{U_k} \in \cloc {U_k}{U_k \cap Y}w$.  Then, the inductive assumption gives us an open set $V_k \Subset_{U_k \cap Y} U_k$ such that $\gExt^{r-k-1}(H^{-k}(Li^*i_*\cF)|_{V_k}, \kappa_*\cG|_{V_k}) = 0$ for all $\cG \in \tcgg Y{w+1}$.  Note that $V_k \Subset_Y Z$, and let
\[
V = V_0 \cap V_1 \cap \cdots \cap V_k.
\]
Clearly, $V \Subset_Y Z$, and all the $\gExt$-vanishing statements above hold over $V$.  In particular, for $k = 1, \ldots, r-1$, the collection of statements
\[
\gExt^{r-k-1}(H^{-k}(Li^*i_*\cF)|_V, \kappa_*\cG|_V) \simeq
\gHom(H^{-k}(Li^*i_*\cF)|_V[k+1], \kappa_*\cG|_V[r]) = 0
\]
lets us deduce (via a standard argument by induction on the number of nonzero $H^{-k}(Li^*i_*\cF)|_V$) that $\gHom((\tau^{[-r+1,-1]}Li^*i_*\cF|_V)[1], \kappa_*\cG|_V[r]) = 0$.  Since we have $\kappa_*\cG|_V[r] \in \dgbg V{-r}$, it follows that
\[
\gHom((\tau^{\le -1}Li^*i_*\cF|_V)[1], \kappa_*\cG|_V[r]) = 0.
\]
Next, we know that $\gExt^r(i_*\cF|_{V_0}, i_*\kappa_*\cG|_{V_0}) \simeq \gHom(Li^*i_*\cF|_{V_0\cap Z}, \kappa_*\cG|_{V_0 \cap Z}[r]) = 0$, so
\[
\gHom(Li^*i_*\cF|_V, \kappa_*\cG|_V[r]) = 0
\]
as well.  Note that $\tau^{\ge 0}Li^*i_*\cF|_V \simeq i^*i_*\cF|_V \simeq \cF|_V$, so we have a distinguished triangle
\[
\tau^{\le -1}Li^*i_*\cF|_V \to Li^*i_*\cF|_V \to \cF|_V \to.
\]
From the long exact sequence of $\gHom$-groups associated to this distinguished triangle and the two vanishing statements established above, we see that
\[
\gHom(\cF|_V, \kappa_*\cG|_V[r]) \simeq \gExt^r(\cF|_V, \kappa_*\cG|_V) = 0
\]
for any $\cG \in \tcgg Y{w+1}$, as desired. 
\end{proof}

We conclude with a basic fact about the structure sheaf of a scheme with
an $s$-structure.

\begin{prop}\label{prop:Ox-zero}
If $X$ is endowed with an $s$-structure, then $\cO_X \in
\cgl X0$, but $\cO_X \notin \cgl X{-1}$.  If $X$ is reduced,
then $\cO_X$ is pure of step $0$.
\end{prop}
\begin{proof}
By axiom~\sref{ax:bdd}, there must be some integer $v$
such that $\cO_X \in \cgl Xv$.  It follows that $\cO_X|_U \in \cgl Uv$ for
every open subscheme $U \subset X$.  Now, suppose that $\cO_X \in \cgg X1$.
Then $v \ge 1$, so by axiom~\sref{ax:tensorr}, we have that $\cO_X^{\otimes v+1} \in
\tcgg X{v+1}$, so there is some dense open subscheme $U$ with $\cO_X^{\otimes
v+1}|_U \in \cgg U{v+1}$.  But of course $\cO_X^{\otimes v+1} \simeq
\cO_X$, so this is a contradiction, and $\cO_X \notin \cgg X1$.

Next, consider $\sigma_{\ge 1}\cO_X$.  This is a quotient of $\cO_X$, and
thus
isomorphic to $i_*\cO_Z$ for some closed subscheme $i: Z \hookrightarrow
X$. Since $i_*$ is $s$-exact, we must have $\cO_Z \in \cgg Z{1}$, but this
is impossible by the previous paragraph, so it follows that $\sigma_{\ge
1}\cO_X = 0$.  Thus, $\cO_X \in \cgl X0$. 

Now, let $v$ be the smallest integer such that $\cO_X \in \cgl Xv$. 
We have just seen that $v \le 0$.  On the other hand, we cannot have $v <
0$, for in that case, axiom~\sref{ax:tensorl} would give us
that 
\[
\cO_X \simeq \cO_X \otimes \cO_X \in \cgl X{2v}.
\]
So $v = 0$, and $\cO_X \notin \cgl X{-1}$.

Finally, consider the subsheaf $\sigma_{\le -1}\cO_X \subset \cO_X$.  Over any open subscheme $U \subset X$, the restriction $\sigma_{\le -1}\cO_X|_U$ is a proper subsheaf of $\cO_U$, since (by the above argument) $\cO_U \notin \cgl U{-1}$.  This condition implies that $\sigma_{\le -1}\cO_X$ is a
nilpotent ideal sheaf.  In particular, if $X$ is reduced, $\sigma_{\le
-1}\cO_X = 0$.
\end{proof}

\section{$s$-structures on Closed Subschemes}
\label{sect:closed}

Throughout this section, we fix the following conventions: let
\[
i: Z \hto X
\]
denote the inclusion of a closed subscheme, and let $\cI_Z \subset \cO_X$ denote the corresponding ideal sheaf.  We assume that $Z$ is endowed with a pre-$s$-structure.  (We do not assume that $X$ has a pre-$s$-structure.)  In addition, we impose the assumption ({\it cf.}~axiom~\aref{ax:ideal} in Definition~\ref{defn:adh}) that
\begin{equation}\label{eqn:pre-adh}
i^*\cI_Z \in \cgl Z0.
\end{equation}
We will frequently consider a second closed subscheme $Z' \subset X$ with the same underlying topological space as $Z$, such that there is an inclusion map
\[
t: Z \hto Z'.
\]
In this setting, we let $\cI \subset \cO_{Z'}$ denote the ideal sheaf of $Z \subset Z'$.  Note that $t^*\cI$ is naturally a quotient of $i^*\cI_Z$, so it follows from~\eqref{eqn:pre-adh} that $t^*\cI \in \cgl Z0$ as well.

One goal of this section is to show that the given pre-$s$-structure on $Z$ induces pre-$s$-structures on all such ``larger'' subschemes $Z'$.  This allows us to define a full subcategory of $\cg X$ as follows:
\[
\czsupp_{\ge w} = \left\{ \cF \,\bigg|
\begin{array}{c}
\text{there is a subscheme structure $i': Z' \hto X$ on $\uZ$} \\
\text{and a sheaf $\cF_1 \in \cgg {Z'}w$ such that $\cF \simeq i'_*\cF_1$}
\end{array} \right\}
\]
Recall also the definition of $\czloc w$ from~\eqref{eqn:cloc-defn}.  Another
result of this section is that $\czloc w$ is a Serre subcategory of $\cg
X$.  These two categories play important and complementary roles in the gluing theorem of Section~\ref{sect:glue}.  Indeed, in the special case where $X$ is simply another scheme structure on $\uZ$, most of the gluing theorem is contained in the results of this section.

\begin{prop}\label{prop:res-nonred-pre}
There is a unique pre-$s$-structure on $Z'$ such that
$t_*$ is $s$-exact.  It is given by: 
\begin{equation}\label{eqn:res-nonred}
\begin{aligned}
\cgl{Z'}w &= \{ \cF \mid t^*\cF \in \cgl Zw \}, \\
\cgg{Z'}w &= \{ \cG \mid \text{$\Hom(\cF,\cG) = 0$ for all $\cF \in
\cgl{Z'}w$} \}.
\end{aligned}
\end{equation}
Moreover, $t^*$ is right $s$-exact, and $t^\flat$ is left $s$-exact.
\end{prop}
\begin{proof}
We begin by proving that $\cgl{Z'}w$ is a Serre subcategory of $\cg{Z'}$. 
It is obviously stable under extensions and quotients, because $\cgl Zw$
is, and $t^*$ is a right exact functor.  It remains to show that it is
closed under subobjects. 

Note that $\cI \subset \cO_{Z'}$ is a nilpotent ideal sheaf, so for any $\cF \in \cg {Z'}$, there is an integer
$n$ such that $\cI^n\cF = 0$.  Let $\ell(\cF)$ be the smallest such
integer.  Now, suppose $\cF \in \cgl {Z'}w$.  We will prove by induction on
$\ell(\cF)$ that all subsheaves of $\cF$ are also in $\cgl{Z'}w$.  If
$\ell(\cF) \le 1$, then $\cF$ is already supported on $Z$: {\it i.e.}, $\cF
\simeq t_*\cF_1$ for some $\cF_1 \in \cg Z$.  Moreover, we necessarily have
$\cF_1 \in \cgl Zw$, since $t^*\cF \simeq \cF_1$.  Every subsheaf of $\cF$
is in $\cgl {Z'}w$ because every subsheaf of $\cF_1$ is in $\cgl Zw$. 

Next, if $\ell(\cF) > 1$, form the short exact sequence
\begin{equation}\label{eqn:ell-ind}
0 \to \cI\cF \to \cF \to \cF/\cI\cF \to 0.
\end{equation}
It is clear that $\ell(\cI\cF) = \ell(\cF) - 1$, and that $\ell(\cF/\cI\cF)
= 1$.  Note that $\cF/\cI\cF \simeq t_*t^*\cF$ lies in $\cgl {Z'}w$ by assumption.  On the other hand, $\cI\cF$ is a quotient of $\cI \otimes \cF$, and we have $t^*(\cI \otimes \cF) \simeq t^*\cI \otimes t^*\cF \in \cgl Zw$ by condition~\eqref{eqn:pre-adh}.  Any subsheaf of $\cF$ is an extension of a subsheaf of $\cF/\cI\cF$ by a subsheaf of $\cI\cF$.  The latter two are in $\cgl {Z'}w$,
and since we already know that $\cgl{Z'}w$ is stable under extensions, we
conclude that any subsheaf of $\cF$ is in $\cgl{Z'}w$.  It is clear that
$\cgg{Z'}w$ is stable under subobjects and extensions, so
axiom~\sref{ax:serre} holds.  We have $\cgg{Z'}w = \cgl{Z'}w^\perp$ by
definition, so by Proposition~\ref{prop:noeth-orth}, axiom~\sref{ax:ses}
holds.  Axiom~\sref{ax:inc} is obvious. 

The idea of the induction argument above will be reused several times.  The
categories $\cg{Z'}$ and $\cgl{Z'}w$ (and later, in Proposition~\ref{prop:res-nonred}, $\tcgg{Z'}w$ as well) are Serre
subcategories of $\cg{Z'}$, so given a sheaf $\cF$ in any of these categories, we can
always form the short exact sequence~\eqref{eqn:ell-ind}, and the sheaves
$\cI\cF$ and $\cF/\cI\cF$ both lie in the same category.  Given any
property that is stable under extensions, we can prove that it holds for
all sheaves in $\cg{Z'}$ (resp.~$\cgl{Z'}w$, $\tcgg{Z'}w$) simply by
proving it for sheaves with $\ell(\cF) = 1$, {\it i.e.}, for sheaves of the
form $\cF = t_*\cF_1$ with $\cF_1 \in \cg Z$.  (Note that this method
cannot be used to prove statements about $\cgg{Z'}w$, because for $\cF \in
\cgg{Z'}w$, it may not be true that $\cF/\cI\cF \in \cgg{Z'}w$.)

Now, fix a sheaf $\cG \in \cgg{Z'}w$.  To show that $\gHom(\cF,\cG) = 0$
for all $\cF \in \cgl{Z'}w$, it suffices to show, by the previous
paragraph, that $\gHom(t_*\cF_1,\cG) = 0$ for all $\cF_1 \in \cgl Zw$. 
Since $\Hom(t_*\cF_1,\cG) \simeq \Hom(\cF_1, t^\flat\cG) = 0$ for all
$\cF_1 \in \cgl Zw$, we see that $t^\flat\cG \in \cgg Zw$.  Now, for any
open subscheme $V \subset Z'$, we have $\Hom((t_*\cF_1)|_V, \cG|_V) \simeq
\Hom(\cF_1|_{V_1 \cap Z}, (t^\flat\cG)|_{V_1 \cap Z})$, and the latter is
$0$ because $\gHom(\cF_1,t^\flat\cG) = 0$.  Therefore, $\gHom(t_*\cF_1,\cG)
= 0$, and axiom~\sref{ax:hom} holds.

Given $\cG \in \cg{Z'}$, consider the sheaf $t^\flat\cG \in
\cg Z$.  By axiom~\sref{ax:bdd} for $Z$, we have $t^\flat\cG \in \cgg Zw$
for some $w$, so $\Hom(\cF_1,t^\flat\cG) = 0$ for all $\cF_1 \in \cgl
Z{w-1}$. It follows that $\Hom(t_*\cF_1, \cG) = 0$, and then, by induction
on $\ell(\cF)$, that $\Hom(\cF,\cG) = 0$ for all $\cF \in \cgl {Z'}{w-1}$. 
Therefore, $\cG \in \cgg {Z'}w$.  On the other hand, there is some $v \in
\Z$ such that $t^*\cG \in \cgl Zv$, and it follows that $\cG \in \cgl
{Z'}v$.  Thus, axiom~\sref{ax:bdd} holds for $Z'$.

Next, fix a sheaf $\cG \in \cgl{Z'}v$.  To verify axiom~\sref{ax:tensorl},
it suffices to show that $t_*\cF_1 \otimes \cG \in \cgl{Z'}{w+v}$ for all
$\cF' \in \cgl Zw$.  This is straightforward: we have $t_*\cF_1 \otimes \cG
\simeq t_*(\cF_1 \otimes t^*\cG)$, and since $t^*\cG \in \cgl Zv$, we know
that $\cF' \otimes t^*\cG \in \cgl Z{w+v}$.  This completes the proof that the categories above constitute a pre-$s$-structure.

We saw earlier that $\cG \in \cgg{Z'}w$ implies $t^\flat\cG \in \cgg Zw$, so $t^\flat$ is left $s$-exact.  It is obvious, from the definition of $\cgl{Z'}w$, that $t^*$ is right $s$-exact.  Finally, we consider the uniqueness statement: suppose we had another collection of subcategories $(\{\cC'_{\le w}\}, \{\cC'_{\ge w}\})$ that formed a pre-$s$-structure and with respect to which $t_*$ was $s$-exact.  For $\cF \in \cgl{Z'}w$, one sees by induction an $\ell(\cF)$ that $\cF \in \cC'_{\le w}$, so $\cgl {Z'}w \subset \cC'_{\le w}$.  But since $\cC'_{\le w}$ is itself a Serre subcategory of $\cg{Z'}$, the same argument goes through with the roles of $\cgl{Z'}w$ and $\cC'_{\le w}$ reversed: we thus obtain $\cC'_{\le w} \subset \cgl{Z'}w$.  Thus, $\cC'_{\le w} = \cgl{Z'}w$, and then the equality $\cC'_{\ge w} = \cgg{Z'}w$ follows from the fact that $\cC'_{\ge w} = \cC'_{\le w-1}{}^\perp$ and $\cgg{Z'}w = \cgl{Z'}{w-1}^{\perp}$.
\end{proof}

The uniqueness statement in Proposition~\ref{prop:res-nonred-pre} implies that the condition in the definition of $\czsupp_{\ge w}$ may be tested on \emph{any} subscheme structure on which a given sheaf is supported.  That is, given $\cF \in \cg X$, suppose we have two subscheme structures $i': Z' \hto X$ and $i'': Z'' \hto X$ on $\uZ$, and sheaves $\cF' \in \cg{Z'}$, $\cF'' \in \cg{Z''}$ such that $\cF \simeq i'_*\cF' \simeq i''_*\cF''$.  Let $Y = Z' \cup Z''$ (that is, if $Z'$ and $Z''$ correspond to ideal sheaves $\cI'$ and $\cI''$, respectively, then $Y$ is the subscheme structure on $\uZ$ corresponding to $\cI' \cap \cI''$).  Then the pre-$s$-structure on $Y$ induced from $Z'$ coincides with that induced from $Z''$, because both of them must agree with the one induced from $Z$.  By comparing the push-forwards of $\cF'$ and $\cF''$ to $Y$, we see that $\cF' \in \cgg{Z'}w$ if and only if $\cF'' \in \cgg{Z''}w$.

\begin{prop}\label{prop:cloc-serre}
For each $w \in \Z$, $\czloc w$ is a Serre
subcategory of $\cg X$.
\end{prop}
\begin{proof}
Because $i^*$ is a right exact functor, and $\cgl Zw$ is closed under
quotients and extensions, it is clear that $\czloc w$ is closed under
quotients and extensions as well.  It remains to show that it is closed
under subobjects.  Note first that for any $k \ge 1$, the sheaf $\cF/\cI^k\cF$
is annihilated by some power of $\cI$, and therefore supported on some
subscheme structure $i_{Z'}:Z' \hto X$ on $\uZ$. 
Indeed, if $\cF/\cI^k\cF \simeq i_{Z'*}\cF_1$, then we must have $\cF_1 \in
\cgl {Z'}w$.  We know that every
subsheaf of $\cF_1$ lies in $\cgl{Z'}w$.  It follows that every subsheaf of
$\cF/\cI^k\cF$ lies in $\czloc w$. 

Now, consider a subsheaf $\cF' \subset \cF$.  Let $\cF'' = \cF'/\cI\cF'
\simeq i_*i^*\cF'$.  Clearly, $i^*\cF'' \simeq i^*\cF'$, so it suffices to
show that $\cF'' \in \czloc w$.  By the Artin--Rees lemma for coherent
sheaves on a noetherian scheme, there exists an integer $k \ge 1$ such that
for all $n \ge k$, we have 
\[
\cI^n\cF \cap \cF' = \cI^{n-k}(\cI^k\cF \cap \cF').
\]
Let us take $n = k + 1$.  Obviously
\[
\cI(\cI^k\cF \cap \cF') \subset \cI \cF' \subset \cF',
\]
so $\cF'/\cI\cF'$ is a quotient of the sheaf $\cG = \cF'/(\cI(\cI^k\cF \cap
\cF'))$.  But then $\cG \simeq \cF'/(\cI^{k+1}\cF \cap \cF')$, so $\cG$ can
be identified with a subsheaf of $\cF/\cI^{k+1}\cF$.  By the previous
paragraph, $\cG \in \czloc w$.  Since $\czloc w$ is closed under quotients,
we see that $\cF'' \in \czloc w$ as well. 
\end{proof}

For the remainder of the section, we study almost $s$-structures that obey an additional condition:

\begin{defn}\label{defn:weak-s}
An almost $s$-structure on $X$ is said to be a \emph{weak $s$-structure} if
for every $\cF \in \cgl Xw$ and every $\cG \in \tcgg X{w+1}$, there is an open subscheme $V \Subset
X$ such that $\gExt^1(\cF|_V, \cG|_V) = 0$.
\end{defn}

Obviously, this condition is a special case of axiom~\sref{ax:ext2}, so every $s$-structure is a weak $s$-structure.  However, it turns out that weak $s$-structures are easier to study in the context of the construction of Proposition~\ref{prop:res-nonred-pre}.  In Section~\ref{sect:glue}, additional hypotheses will allow us to obtain $s$-structures, but we will require the following results on weak $s$-structures along the way.

\begin{lem}\label{lem:ext1-nonred}
Assume that $Z$ is endowed with a weak $s$-structure.  Then, for any $\cF \in \cgl Zw$ and any $\cG \in \tcgg Z{w+1}$, there is an
open subscheme $V \Subset Z'$ such that $\gExt^1(t_*\cF|_V, t_*\cG|_V) = 0$. 
\end{lem}
\begin{proof}
From Definition~\ref{defn:weak-s}, there is a dense open subscheme $V_1
\Subset Z$
such that $\gExt^1(\cF|_{V_1}, \cG|_{V_1}) = 0$.  By replacing $V_1$ by a smaller open subscheme if necessary, we may assume that $\cG|_{V_1} \in \cgg {V_1}{w+1}$.  Let $V$ be the corresponding open subscheme of $Z'$.  

We will show that any short exact sequence $0 \to t_*\cG|_V \to \cH \to t_*\cF|_V \to 0$ in $\cg V$ splits.  The same argument will also be valid when $V$ is replaced by any open subscheme of $V$, so it will follow that $\gExt^1(t_*\cF|_V, t_*\cG|_V) = 0$.

For brevity, we will write $\cI$ instead of $\cI|_V$ for the ideal sheaf in $\cO_V$ corresponding to the closed subscheme $V_1 = V \cap Z$.  Let us identify $t_*\cG|_V$ with its image in
$\cH$.  Because $t_*\cF|_V$ is supported on $V_1$, we see that $\cI\cH$ must be
contained in the kernel of the morphism $\cH \to t_*\cF|_V$; {\it i.e.},
$\cI\cH \subset t_*\cG|_V$.  In particular, we see that $\cI\cH \in
\cgg V{w+1}$.  Consider the diagram 
\[
\xymatrix@=10pt{
t_*\cG|_V \otimes \cI \ar[r] &
\cH \otimes \cI \ar[d]\ar[r] & t_*\cF|_V \otimes \cI \ar@{-->}[dl]\ar[r] & 0
\\
& \cI\cH}
\]
The top row is right exact.  Since $t_*\cG|_V$ is supported on $V_1$, it is annihilated by $\cI$, so the natural map
$t_*\cG|_V \otimes \cI \to \cI\cH$ is $0$.  In other words, the image of
$t_*\cG|_V \otimes \cI$ in $\cH \otimes \cI$ is contained in the kernel of the
map $\cH \otimes \cI \to \cI\cH$, so the latter factors through $t_*\cF|_V
\otimes \cI$.  Now, it is obvious from the definition of $\cgl Vw$ that $\cO_V \in \cgl V0$, so its subsheaf $\cI$ is in $\cgl V0$ as well.  (We cannot yet invoke Proposition~\ref{prop:Ox-zero} on $V$, however.)  Therefore,
$t_*\cF|_V \otimes \cI \in \cgl Vw$, so its quotient
$\cI\cH$ must be in $\cgl Vw$ as well.  Since we also have $\cI\cH \in
\cgg V{w+1}$, we conclude that $\cI\cH = 0$, and hence that $\cH$ is
actually supported on $V_1$: we have $\cH \simeq t_*\cH_1$ for some $\cH_1 \in \cg {V_1}$.  The sequence $0 \to
\cG|_{V_1} \to \cH_1|_{V_1} \to \cF|_{V_1} \to 0$ splits because $\gExt^1(\cF|_{V_1}, \cG|_{V_1}) = 0$, and hence so
does $0 \to t_*\cG|_V \to \cH \to t_*\cF|_V \to 0$. 
\end{proof}

\begin{prop}\label{prop:res-nonred}
Assume that $Z$ is endowed with a weak $s$-structure.  Then the induced pre-$s$-structure on $Z'$ is also a weak $s$-structure. 
\end{prop}
\begin{proof}
Suppose $\cG \in \tcgg{Z'}w$. We will simultaneously prove
axiom~\sref{ax:serrer} and the condition in Definition~\ref{defn:weak-s} by induction
on the invariant $\ell(\cG)$ that was introduced in the proof of
Proposition~\ref{prop:res-nonred-pre}.  

If $\ell(\cG) = 1$, then we have $\cG \simeq t_*\cG_1$ for some $\cG_1 \in
\tcgg Zw$, so every quotient of $\cG$ is in $\tcgg{Z'}w$ because every
quotient of $\cG_1$ is in $\tcgg Zw$.  Next, observe that for any $\cF \in
\cgl{Z'}{w-1}$, we have a sequence
\[
\gExt^1(\cF/\cI\cF, t_*\cG_1) \to \gExt^1(\cF, t_*\cG_1) \to
\gExt^1(\cI\cF, t_*\cG_1)
\]
that is exact at the middle term.  Recall that $\cgl{Z'}{w-1}$ is a Serre
subcategory of $\cg{Z'}$.  Thus, an argument by induction on $\ell(\cF)$,
combined with Lemma~\ref{lem:ext1-nonred}, shows that there exists an open
subscheme $V \subset Z'$ with $\gExt^1(\cF|_V, t_*\cG_1|_V) = 0$.

Now, suppose $\ell(\cG) > 1$.  We claim that $\cG/\cI\cG \in \tcgg{Z'}w$.  Suppose instead that $\cG/\cI\cG \notin \tcgg{Z'}w$, and let $\cH = \sigma_{\le w-1}(\cG/\cI\cG)$.  The natural map $\cH \to \cG/\cI\cG$ remains nonzero upon restriction to any dense open subset of $Z'$.  Since $\ell(\cI\cG) < \ell(\cG)$, we know inductively that there is an open subscheme $V_1 \subset Z'$ such that
$\gExt^1(\cH|_{V_1}, \cI\cG|_{V_1}) = 0$.  By replacing $V_1$ with a smaller open subscheme if necessary, we may also assume that $\gHom(\cH|_{V_1}, \cG|_{V_1}) = 0$.  Consider the exact sequence
\[
\gHom(\cH|_{V_1}, \cG|_{V_1}) \to
\gHom(\cH|_{V_1}, (\cG/\cI\cG)|_{V_1}) \to \gExt^1(\cH|_{V_1},
\cI\cG|_{V_1}).
\]
We now have a contradiction: the first and last terms vanish, but the middle term does not.  Thus, it must be that $\cG/\cI\cG \in \tcgg{Z'}w$.
Now, any quotient of $\cG$ is an
extension of a quotient of $\cG/\cI\cG$ be a quotient of $\cI\cG$.  The
latter two are now known to be in $\tcgg{Z'}w$, so we see that any
quotient of $\cG$ is in $\tcgg{Z'}w$.

Next, let $\cF \in \cgl{Z'}{w-1}$.  By invoking the inductive assumption twice, we may find a dense open subscheme $V \Subset Z'$ such that the first and last terms of the following exact sequence both vanish:
\[
\gExt^1(\cF|_V, \cI\cG|_V) \to \gExt^1(\cF|_V, \cG|_V) \to \gExt^1(\cF|_V,
(\cG/\cI\cG)|_V).
\]
We then see that $\gExt^1(\cF|_V, \cG|_V) = 0$ as well, as desired.

We now turn to axiom~\sref{ax:tensorr}.  Because
$\tcgg{Z'}w$ is now known to be a Serre subcategory, this axiom can
be deduced by induction on $\ell(\cF)$ from the corresponding statement on $Z$,
using the same argument that was given for axiom~\sref{ax:tensorl} in the
proof of Proposition~\ref{prop:res-nonred-pre}. 
\end{proof}

\section{The Gluing Theorem for $s$-structures}
\label{sect:glue}

Given an $s$-structure on a closed subscheme of $X$, and another $s$-structure on its open complement, we show in this section how to produce an $s$-structure on $X$.  The latter is said to be obtained by \emph{gluing} the two given $s$-structures.  Gluing cannot be carried out for arbitrary $s$-structures, however; certain compatibility conditions must hold.

In particular, consider the following condition on a closed subscheme:

\begin{defn}\label{defn:adh}
Let $Z \subset X$ be a closed subscheme with a pre-$s$-structure.
$Z$ is said to be \emph{adhesive} if it satisfies
the following two conditions: 
\begin{enumerate}
\renewcommand{\labelenumi}{(A\arabic{enumi})}
\item The ideal sheaf $\cI \subset \cO_X$ corresponding to $Z$ is in $\czloc 0$.\label{ax:ideal}
\item For any closed subscheme $Y \subset Z$, any $\cF \in \cloc XYw$, and any $r \ge 0$, there is a dense open subscheme $V \Subset_Y X$ such that $\gExt^r(\cF|_V, \cG|_V) = 0$ for all $\cG \in \csupp XY_{\ge w+1}$.\label{ax:adh-ext2} 
\end{enumerate}
\end{defn}

Note that axiom~\aref{ax:ideal} is the same as condition~\eqref{eqn:pre-adh}, so the results of Section~\ref{sect:closed} can be applied to adhesive subschemes.
Here, $X$ is not assumed to have an $s$-structure, but if $X$ does happen to have one, it is easy to see that every closed subscheme (with the induced $s$-structure) is automatically adhesive.  Indeed, it follows from Proposition~\ref{prop:Ox-zero} that $\cI \in \cgl X0$, and hence, from the description of the $s$-structure on $Z$ in Proposition~\ref{prop:res-closed}, that $\cI \in \czloc 0$.  Axiom~\aref{ax:adh-ext2} is merely a special case of axiom~\sref{ax:ext2} for the $s$-structure on $X$.  Thus, adhesiveness of the closed subscheme must be a necessary condition for a gluing theorem to hold.

We will begin with a preliminary version, in which we only produce a pre-$s$-structure.

\begin{prop}\label{prop:glue-pre}
Let $i: Z \hto X$ be a closed subscheme, and let $j: U \to X$ be
the complementary open subscheme.  Suppose that $Z$ is endowed with an $s$-structure, and that $U$ is endowed with a pre-$s$-structure.  Define two full subcategories of $\cg X$ by
\begin{equation}\label{eqn:glue-s-struc}
\begin{aligned}
\cgl Xw &= \{ \cF \in \cg X \mid 
\text{$j^*\cF \in \cgl Uw$ and $\cF \in \czloc w$} \}, \\
\cgg Xw &= \{ \cG \in \cg X \mid 
\text{$j^*\cG \in \cgg Uw$ and $\Gamma_Z\cG \in \czsupp_{\ge w}$}
\}.
\end{aligned}
\end{equation}
Assume that $Z$ is adhesive, and that the following condition holds:
\begin{equation}\label{it:lower-extend}
\text{\rm For every $\cF' \in \cgl Uw$, there exists some $\cF \in \cgl
Xw$ with $j^*\cF \simeq \cF'$.}
\end{equation}
Then, the above categories define the
unique pre-$s$-structure on $X$ whose restrictions to $U$ and $Z$
coincide with the given ones.
\end{prop}
\begin{proof}
We proceed by noetherian induction: we assume that for any closed subscheme
of $X$, the theorem is already known.  The base cases are those in which
either $Z$ or $U$ is empty.  The theorem is true in both these cases: the
former is trivial, and the latter reduces to
Proposition~\ref{prop:res-nonred-pre}.  Henceforth, assume that $U$ and $Z$ are both nonempty.

It is easy to see that $\cgl Xw$ is a Serre subcategory of $\cg X$, using Proposition~\ref{prop:cloc-serre}, the exactness of $j^*$, and the fact that $\cgl Uw$ is a Serre subcategory of $\cg U$.  Similarly, $\cgg Xw$ is closed under subobjects and extensions, because $\cgg Uw$ and $\czsupp_{\ge w}$ are, and $j^*$ and $\Gamma_Z$ are left exact.

Next, we consider axiom~\sref{ax:hom}. Suppose $\cF \in \cgl Xw$ and $\cG
\in \cgg X{w+1}$, and let $V \subset X$ be an open subscheme.  There is
some closed subscheme structure $i_{Z'}: Z' \hto X$ on the underlying space
of $Z$ and some sheaf $\cG_1 \in \cgg {Z'}{w+1}$ such that $\Gamma_Z\cG
\simeq i_{Z'*}\cG_1$.  Let $t: Z \hto Z'$ be the inclusion map.  Since
$i^*\cF \simeq t^*i_{Z'}^*\cF$ is assumed to be in $\cgl Zw$, we see that
$i_{Z'}^*\cF \in \cgl{Z'}w$, and then we have 
\[
\Hom(\cF|_V, \Gamma_Z\cG|_V) \simeq \Hom(\cF|_V, i_{Z'*}\cG_1|_V) \simeq
\Hom(i_{Z'}^*\cF|_{V \cap Z'}, \cG_1|_{V \cap Z'}) = 0. 
\]
In addition, we clearly have $\gHom(j^*\cF, j^*\cG) = 0$, so from the exact
sequence
\[
0 \to \Hom(\cF|_V, \Gamma_Z\cG|_V) \to \Hom(\cF|_V,\cG|_V) \to
\Hom(j^*\cF|_{V \cap U},j^*\cG|_{V \cap U}), 
\]
we see that $\Hom(\cF|_V,\cG|_V) = 0$, and hence $\gHom(\cF,\cG) = 0$.

Next, suppose $\cG \in \cgl Xw^\perp$.  We will show that if $U' \subset X$ is any open subscheme containing $U$,
then we have $\Hom(\cF|_{U'}, \cG|_{U'}) = 0$ for all $\cF \in \cgl Xw$.  (This fact will be used to establish axiom~\sref{ax:ses}.)
We
proceed by noetherian induction on the complement of $U'$.  In the case $U'
= X$, there is nothing to prove; $\Hom(\cF,\cG) = 0$ by assumption.  Now,
suppose $U' \ne X$.  Let $Y$ be any closed subscheme structure on
the complement $X \ssm U'$ that also makes it into a closed subscheme of $Z$.  Then, let
$i_{Y'}: Y' \hto X$ be a closed subscheme structure on the underlying space
of $Y$ such that $\Gamma_Y\cG$ is supported on $Y'$: we have $\Gamma_Y\cG
\simeq i_{Y'*}\cG_1$ for some $\cG_1 \in \cg {Y'}$.  Now, from the
injective map 
\[
\Hom(\cF, \Gamma_Y\cG) \to \Hom(\cF, \cG),
\]
we see that $\Hom(\cF, i_{Y'*}\cG_1) = 0$ for all $\cF \in
\cgl Xw$.  In particular, considering just those $\cF$ of the form $\cF
\simeq i_{Y'*}\cF_1$ with $\cF_1 \in \cgl {Y'}w$, we see that $\cG_1 \in
\cgg {Y'}{w+1}$, and hence $\Gamma_Y\cG \in \csupp XY_{\ge w+1}$.

Therefore, because $Z$ is adhesive, there is an open subscheme
$U'' \Subset_Y X$ such that
$\Ext^1(\cF|_{U''}, \Gamma_Y\cG|_{U''}) = 0$.  Form the long exact
sequence 
\begin{multline*}
0 \to \Hom(\cF|_{U''}, \Gamma_Y\cG|_{U''}) \to \Hom(\cF|_{U''},\cG|_{U''})
\to \\
\Hom(\cF|_{U''},\cG/\Gamma_Y\cG|_{U''}) \to \Ext^1(\cF|_{U''},
\Gamma_Z\cG|_{U''}) \to \cdots.
\end{multline*}
Recall our inductive assumption: $U''$ is strictly larger than $U'$ (since
$U'' \cap Y$ is dense in $Y$), so we assume that
$\Hom(\cF|_{U''},\cG|_{U''})$ is known to be $0$ for all $\cF \in \cgl Xw$.
 It follows that $\Hom(\cF|_{U''},\cG/\Gamma_Y\cG|_{U''}) = 0$.  The same
holds if we replace $\cF$ by any subsheaf (since $\cgl Xw$ is closed under
subobjects), so by Lemma~\ref{lem:open0}, we have that $\Hom(\cF|_{U'},
\cG|_{U'}) = 0$, as desired. 

Thus, $\Hom(\cF|_{U'}, \cG|_{U'}) = 0$ for any open subscheme
containing $U$ and any $\cF \in \cgl Xw$.  In particular, this holds for
$U' = U$.  By
assumption, every sheaf in $\cgl Uw$ occurs as the restriction to $U$ of
some $\cF \in \cgl Xw$, so the fact that $\Hom(j^*\cF,j^*\cG) = 0$ for all
$\cF \in \cgl Xw$ implies that $j^*\cG \in \cgg U{w+1}$.  On the other
hand, we saw in the preceding paragraph that $\Gamma_Y \cG \in \csupp
XY_{\ge w+1}$ for any closed subscheme $Y \subset Z$.  This holds in
particular for $Y = Z$.  We are able to conclude, at long last, that $\cG
\in \cgl Xw^\perp$ implies that $\cG \in \cgg X{w+1}$. 

Since axiom~\sref{ax:hom} is already established, we see that $\cgl
Xw^\perp = \cgg X{w+1}$, and then by Proposition~\ref{prop:noeth-orth},
axiom~\sref{ax:ses} holds. 

Given $\cF \in \cg X$, there is some $w_1 \in \Z$ such that $j^*\cF \in
\cgg U{w_1}$.  On the other hand, suppose $\Gamma_Z\cF \simeq
i_{Z'*}\cF_1$, where $i_{Z'}: Z' \hto X$ is a closed subscheme structure
on $\uZ$.  There is some $w_2$ such that $\cF_1 \in \cgg{Z'}{w_2}$.  We
then clearly have $\cF \in \cgg X{\min\{w_1,w_2\}}$.  Similarly, we know
that $j^*\cF \in \cgl U{v_1}$ and $i^*\cF \in \cgl Z{v_2}$ for some $v_1,
v_2 \in \Z$, so $\cF \in \cgl X{\max\{v_1,v_2\}}$.  Thus,
axiom~\sref{ax:bdd} holds.

Axiom~\sref{ax:tensorl} is immediate from the fact that $j^*(\cF \otimes
\cG) \simeq j^*\cF \otimes j^*\cG$ and $i^*(\cF \otimes \cG) \simeq i^*\cF
\otimes i^*\cG$. 

To show uniqueness, suppose there were another pre-$s$-structure
$(\{\cC'_{\le w}\}, \{\cC'_{\ge w}\})$ on $X$ whose restrictions to $U$ and
$Z$ were the given $s$-structures on those schemes.  Since $j^*$ and $i^*$
must both be right $s$-exact, we see that $\cF \in \cC'_{\le w}$ implies that $j^*\cF \in
\cgl Uw$ and $i^*\cF \in \cgl Zw$. 
In other words, $\cC'_{\le w} \subset \cgl Xw$.  On the other hand, for any $\cG \in \cC'_{\ge w+1}$, there is some subscheme structure $i_{Z'}: Z' \hto X$ and some sheaf $\cG_1 \in \cg{Z'}$ such that $\Gamma_Z\cG \simeq i_{Z'*}i_{Z'}^\flat \cG$.  Since $i_{Z'}^\flat$ is left $s$-exact, we see that $\Gamma_Z\cG \in \czsupp_{\ge w+1}$.  The functor $j^*$ is also left $s$-exact, so in fact we must have $\cG \in \cgg X{w+1}$.  We deduce that $\cC'_{\ge w}
\subset \cgl Xw$.  Now, the equalities $\cC_{\le w}^{\prime\perp} =
\cC'_{\ge w+1}$ and $\cgl Xw^\perp = \cgg X{w+1}$ together imply that
$\cC'_{\le w} = \cgl Xw$ and $\cC'_{\ge w} = \cgg Xw$.
\end{proof}

The main result of this section is the following.

\begin{thm}\label{thm:glue}
Let $i: Z \hto X$ be a closed subscheme of $X$, and let $j: U \hto
X$ be the complementary open subscheme.  Assume that $U$ and $Z$ are both endowed with $s$-structures, that $Z$ is adhesive, and that condtion~\eqref{it:lower-extend} holds.  Then the pre-$s$-structure on $X$ obtained by gluing is an $s$-structure.
\end{thm}
\begin{proof}
Let $\kappa: Y \hto X$ be a closed subscheme.  Let $U_1 = Y \cap U$, and let $V_1 \subset Y$ be the open subscheme complementary to $Y \ssm \overline U_1$.  Clearly, $U_1 \cup V_1$ is a dense open subscheme of $Y$.  Now, $U_1$ carries an $s$-structure because it is a closed subscheme of $U$.  On the other hand, $V_1$ is a possibly nonreduced scheme whose subscheme $V_1 \cap Z$ has the same underlying topological space.  $V_1 \cap Z$ carries an $s$-structure induced from $Z$, so by Proposition~\ref{prop:res-nonred}, $V_1$ has at least a weak $s$-structure.  Therefore, $U_1 \cup V_1$ carries a weak $s$-structure.

Let $h: U_1 \cup V_1 \hto Y$ be the inclusion map.  We know that $h^*$ is exact, and by Lemma~\ref{lem:res-open-tcgg}, we know that $\cF \in \tcgg Yw$ if and only if $h^*\cF \in \tcgg{U_1 \cup V_1}w$.  These observations, combined with the formula $h^*(\cF \otimes \cG) \simeq h^*\cF \otimes h^*\cG$, imply that axioms~\sref{ax:serrer} and~\sref{ax:tensorr} for $Y$ follow from the fact that $U_1 \cup V_1$ has an almost $s$-structure.

Finally, we consider axiom~\sref{ax:ext2}.  Suppose $\cF \in \cgl Xw$, and fix an integer $r \ge 0$.  Since $U$ has an $s$-structure, we know that there is an open subscheme $V' \Subset_{U_1} U$ such that $\gExt^r(\cF|_{V'}, \kappa_*\cG|_{V'}) = 0$ for all $\cG \in \tcgg Y{w+1}$.  Next, we would like to find an open subscheme $V'' \Subset_{V_1} X \ssm \overline U_1$ such that $\gExt^r(\cF|_{V''}, \kappa_*\cG|_{V''}) = 0$.  If we had such a $V''$, note that it would be open as a subscheme of $X$ and disjoint from $V'$.  Their union $V = V' \cup V''$ (which would satisfy $V \Subset_Y X$) would then be an open subscheme of the sort required by axiom~\sref{ax:ext2} for $Y$.  

Let $V_2 = V_1 \cap Z$.  Note that $V_2$ is a closed subscheme of $V_1$ with the same underlying topological space as $V_1$.  $V_2$ is also a locally closed subscheme of $Z$.  We have $\kappa_*\cG|_{X \ssm \overline U_1} \in \csupp {X\ssm \overline U_1}{V_2}_{\ge w+1}$, so the desired open subscheme $V''$ comes from condition~\aref{ax:adh-ext2} of Definition~\ref{defn:adh} applied to the closed subscheme $\overline V_2 \subset Z$.
\end{proof}

\section{Duality}
\label{sect:duality}

This section is devoted to studying the interaction between $s$-structures and Grothendieck--Serre duality.  The results established here will be essential to the construction of the staggered $t$-structure in Sections~\ref{sect:stag-orth} and~\ref{sect:stag-dt}.

Throughout, we assume that $X$ is endowed with an $s$-structure.  Fix an equivariant dualizing complex $\omega_X$ on $X$, and recall that we denote the
Grothendieck--Serre duality functor by $\D = \cRHom(\cdot, \omega_X)$.

For any point $x\in X$, let $k(x)$ be the residue field of the local ring
$\cO_x$, and let $I_x$ be the injective hull of $k(x)$ in the category
$\Oxmod$ of $\cO_x$-modules.  Recall that for any complex of coherent
sheaves $\cF$ on $X$, the local cohomology groups $H^i_x(\cF) =
H^i(R\Gamma_x(\cF))$ are artinian $\cO_x$-modules.  Here, and throughout the remainder of the paper, points (in contrast to subschemes) are not to be thought of in any equivariant sense.  Accordingly, $\cO_x$-modules are modules without any group action, and $R\Gamma_x$ is to be computed after passing to the nonequivariant derived category.

\begin{lem}\label{lem:omega-conc}
At any point $x \in X$, the complex $R\Gamma_x(\omega_X)$ is
concentrated in one degree $d$.  In addition, if $x$ is a generic point of $X$, there is an open subscheme $U$ containing $x$ such that $\omega_X|_U$ is concentrated
in degree $d$.
\end{lem}
\begin{proof}
The fact that $R\Gamma_x(\omega_X)$ is concentrated in a single degree follows from~\cite[Proposition~V.3.4]{har}.  When $x$ is generic, the existence of the desired open subscheme $U$ is given by~\cite[Lemma~2]{bez:pc}.
\end{proof}

For the remainder of the section, we will usually assume that $\omega_X$ is a sheaf, {\it i.e.}, a complex concentrated in degree $0$.  According to the previous lemma, this situation can always be achieved by replacing $X$ by a suitable open subscheme, and by shifting $\omega_X$.

\begin{lem}\label{lem:wt-defn}
Assume $\omega_X$ is a sheaf.  There is an open subscheme $U
\subset X$ and an integer $e$ such that for all open
subschemes $V \subset U$,
$\omega_X|_V \in \cgg Ve$ and $\omega_X|_V \notin \cgg V{e+1}$. 
\end{lem}
\begin{proof}
By axiom~\sref{ax:bdd}, there is a $v$ such that $\omega_X \in
\cgl Xv$.  Choose a generic point $x \in X$.  Consider the
sequence of subobjects 
\[
\cdots \subset \sigma_{\le v-1}\omega_X \subset
\sigma_{\le v}\omega_X = \omega_X.
\]
Applying the left exact functor $\Gamma_x$ gives us a decreasing sequence
of
subobjects of $\Gamma_x\omega_X \simeq I_x$, and the latter is an artinian
$\cO_x$-module.  Thus, there is
some integer $m$ such that $\Gamma_x\sigma_{\le w} \omega_X = \Gamma_x
\sigma_{\le m} \omega_X$ for all $w \le m$.  Indeed, it follows then that 
\[
\Gamma_x\left( \bigcap_{w \le m} \sigma_{\le w}\omega_X\right) = \Gamma_x
\sigma_{\le m}\omega_X, 
\]
so by axiom~\sref{ax:bdd}, both sides of this equation must
be $0$.  Now, assume that $m$ is in fact the largest integer such that
$\Gamma_x\sigma_{\le m}\omega_X = 0$.  Since $\Gamma_x\sigma_{\le
m}\omega_X =
0$, there is an irreducible open subscheme $U \subset X$ containing $x$
such that $(\sigma_{\le m}\omega_X)|_U = 0$, and hence $(\sigma_{\le
w}\omega_X)|_U = 0$ for all $w \le m$.  On the other hand, the fact that
$\Gamma_x\sigma_{\le m+1}\omega_X \ne 0$ implies that for every open
subscheme
$V$
containing $x$ (and in particular every $V \subset U$), we have
$(\sigma_{\le m+1}\omega_X)|_V \ne 0$.  Thus, the open subscheme $U$ and
the
integer $e = m+1$ have the desired properties. 
\end{proof}

Combining the above two lemmas leads to the following definition.

\begin{defn}\label{defn:alt}
Assume $X$ is irreducible.  The \emph{altitude} of $X$, denoted $\alt X$, is
the unique integer such that for any open subscheme $U \subset X$, there is
an open subscheme $V \subset U$ such that $\omega_X|_V$ is concentrated in
a single degree $d$, and $\omega_X[d]|_V \in \cgg V{\alt X}$ but
$\omega_X|_V
\notin\cgg V{\alt X+1}$. 
\end{defn}

Below, we will see that the altitude of a closed subscheme does not depend on the choice of closed subscheme structure.  This result depends on the following useful lemma.

\begin{lem}\label{lem:dc-sum}
Assume that $\omega_X$ is a sheaf.  If $\omega_X$ has a
subsheaf $\cF$ that is a direct sum of two smaller
subsheaves, $\cF \simeq \cF_1 \oplus \cF_2$, then there is an open
subscheme $U \subset X$ such that either $\cF_1|_U = 0$ or
$\cF_2|_U = 0$. 
\end{lem}
\begin{proof}
Let $x$ be a generic point of $X$.  Applying $\Gamma_x$, we get an
injective map 
\[
\Gamma_x\cF_1 \oplus \Gamma_x\cF_2 \to \Gamma_x\omega_X \simeq I_x.
\]
But every submodule of $I_x$ contains the residue field $k(x)$, so no
submodule of $\Gamma_x\omega_X$ is a direct sum of two smaller nonzero
submodules.  Thus, either $\Gamma_x\cF_1$ or $\Gamma_x\cF_2$ must be $0$,
and therefore
there
is an open subscheme $U \subset Y$ on which either $\cF_1$ or $\cF_2$
vanishes. 
\end{proof}

\begin{prop}\label{prop:alt-defn}
Let $X'$ be a nonreduced irreducible scheme, and let $t: X \hto X'$ be a closed subscheme with the same underlying space.  Then $\alt X = \alt X'$. 
\end{prop}
\begin{proof}
By replacing $X'$ by a suitable open subscheme and shifting if necessary, we
may assume that $\omega_{X'}$ and $\omega_X$ are both sheaves, and that
$\omega_{X'} \in \cgg {X'}{\alt X'}$ and $\omega_X \in
\cgg X{\alt X}$.  Let $\cI \subset \cO_{X'}$ be the nilpotent ideal sheaf corresponding to $X \subset X'$.  Then $\cO_{X'}/\cI \simeq \cO_X$, and $\omega_X \simeq
\cHom(\cO_{X'}/\cI,  \omega_{X'})$.  Thus, $\omega_X$ can naturally be
identified with a subsheaf of $\omega_{X'}$.  It follows that $\alt X \ge
\alt X'$.

Now, suppose that $\alt X > \alt X'$.  Then $\omega_X$, regarded
as a subsheaf of $\omega_{X'}$, has trivial intersection with the nonzero
sheaf $\sigma_{\le \alt X'}\omega_{X'}$.  That is, $\omega_{X'}$ contains the
direct sum 
\[
\sigma_{\le \alt X'}\omega_{X'} \oplus \omega_X
\]
as a subsheaf.  According to Lemma~\ref{lem:dc-sum}, we can replace $X'$ by
an open subscheme and then assume that one of $\sigma_{\le \alt X'}\omega_{X'}$
or $\omega_X$ vanishes.  But this is clearly absurd.  We conclude
that $\alt X = \alt X'$. 
\end{proof}

The last few results in this section will play an essential role in the developments of Sections~\ref{sect:stag-orth} and~\ref{sect:stag-dt}.

\begin{lem}\label{lem:dual-conc}
Assume that $X$ is irreducible and that $\omega_X$ is a sheaf.  For any $\cF \in \cg X$, there is an open subscheme $U \subset X$ such that $\D(\cF)|_U$ is a sheaf.
\end{lem}
\begin{proof}
Let $x$ be a generic point of $X$, and let $\cG = \D(\cF)$.  Then we have
$R\Gamma_x\cG
\simeq \RHom(\cF_x, I_x)$, but since $I_x$ is injective, this is simply
$\Hom(\cF_x,I_x)$.  Since $R\Gamma_x\cG$ is concentrated in degree $0$,
there is an open subscheme $U \subset X$ such that $\cG|_U$
is a sheaf.
\end{proof}

\begin{lem}\label{lem:st-dual}
Assume that $X$ is irreducible and that $\omega_X$ is a sheaf.  Let $\cF
\in \cg X$ be pure of step $w$.  Then there is an open
subscheme $U \subset X$ such that $\D(\cF)|_U$ is a sheaf, and such that it
is pure of step $\alt X - w$.
\end{lem}
(Recall that $\cF$ is pure of step $w$ if it lies in $\cgl Xw \cap
\cgg Xw$.)
\begin{proof}
Let $x$ be a generic point of $X$, and let $\cG = \D(\cF)$.  Then we have
$R\Gamma_x\cG
\simeq \RHom(\cF_x, I_x)$, but since $I_x$ is injective, this is simply
$\cHom(\cF_x,I_x)$.  Since $R\Gamma_x\cG$ is concentrated in degree $0$,
there is an open subscheme $U \subset X$ such that $\cG|_U$
is a sheaf.  Replacing $U$ by a smaller open subscheme if necessary, we may
assume that $\omega_X|_U \in \cgg U{\alt X}$.  By Remark~\ref{rmk:homl}, we then know that $\cG|_U \simeq \cHom(\cF|_U, \omega_X|_U)$
lies in $\cgg U{\alt X - w}$.

Now, consider the sequence
\[
\Gamma_x \sigma_{\le \alt X - w}\cG|_U \subset \Gamma_x\sigma_{\le \alt X - w
+1}\cG|_U \subset \cdots. 
\]
This sequence is eventually constant; let $N \ge \alt X - w$ be the smallest
integer such that 
\[
\Gamma_x \sigma_{\le N}\cG|_U = \Gamma_x \sigma_{\le v}\cG|_U
\]
for all $v \ge N$.  By replacing
$U$ by a smaller open subscheme if necessary, we may assume that
$\sigma_{\ge N+1}\cG|_U = 0$, but $\sigma_{\ge N}\cG|_V \ne 0$ for all open
subschemes $V \subset U$. 

If $N = \alt X - w$, we are done, so assume $N > \alt X - w$.  Now, form the
exact sequence 
\[
0 \to \sigma_{\le N-1}\cG|_U \to \cG|_U \to \sigma_{\ge N}\cG|_U \to 0.
\]
By invoking Lemma~\ref{lem:dual-conc} three times and replacing $U$ by a smaller open
subscheme as necessary, we can assume that the duality functor takes each
term of the above sequence to a sheaf.  (Note that $N$ has been chosen in
such a way that $\sigma_{\ge N}\cG|_U$ remains nonzero as we shrink $U$.)  Thus, $\D$
transforms the exact sequence above into 
\[
0 \to \D(\sigma_{\ge N}\cG)|_U \to \cF|_U \to \D(\sigma_{\le N-1}\cG)|_U
\to 0. 
\]
Now, $\D(\sigma_{\ge N}\cG)|_U$ is a subsheaf of $\cF|_U$, so clearly
$\D(\sigma_{\ge N}\cG)|_U \in \cgg Uw$.  By axioms~\sref{ax:serrer} and~\sref{ax:tensorr}, we
may replace $U$ by a smaller open subscheme and assume that the tensor
product $\D(\sigma_{\ge N}\cG)|_U \otimes \sigma_{\ge N}\cG|_U \simeq \cHom(\sigma_{\ge N}\cG|_U, \omega_X|_U) \otimes \sigma_{\ge N}\cG|_U$ lies in
$\cgg U{w + N}$, as does the image of the evaluation morphism 
\[
\ev: \cHom(\sigma_{\ge N}\cG|_U, \omega_X|_U) \otimes \sigma_{\ge N}\cG|_U
\to \omega_X|_U. 
\]
By assumption, $w + N > \alt X$, so $\omega_X|_U$ contains the direct sum
$\sigma_{\le \alt X}\omega_X|_U \oplus \im(\ev)$.  By Lemma~\ref{lem:dc-sum},
we can replace $U$ by a smaller subscheme and assume that $\im(\ev) = 0$
(clearly, $\sigma_{\le \alt X}\omega_X|_U$ cannot be $0$).  But $\ev = 0$
implies that $\cHom(\sigma_{\ge N}\cG|_U, \omega_X|_U) = 0$, and
$\D(\sigma_{\ge N}\cG)|_U = 0$ in turn implies that $\sigma_{\ge N}\cG|_U =
0$.  This is a contradiction, so we conclude that $N = \alt X - w$, and
$\cG|_U \in \cgl U{\alt X - w} \cap \cgg U{\alt X - w}$. 
\end{proof}

\begin{prop}\label{prop:st-dual}
Assume that $X$ is irreducible and that $\omega_X$ is a sheaf.  Let $\cF
\in \cgg Xw$.  Then there is an open
subscheme $U \subset X$ such that $\D(\cF)|_U$ is a sheaf in $\cgl U{\alt X - w}$.
\end{prop}
\begin{proof}
By axiom~\sref{ax:bdd}, there is a $v \in \Z$ such that $\cF \in \cgl
Xv$. Obviously, we have $v \ge w$.  We proceed by induction on $v - w$.  If
$v - w = 0$, then we are in the situation of Lemma~\ref{lem:st-dual}, so
the statement is already known.  Otherwise, consider the short exact
sequence
\[
0 \to \sigma_{\le w}\cF \to \cF \to \sigma_{\ge w+1}\cF \to 0.
\]
Apply Lemma~\ref{lem:dual-conc} to each term of this sequence: let $U$ be an open subscheme on which $\D$ takes all three terms to sheaves.  Thus, the long exact sequence
\[
0 \to H^0(\D(\sigma_{\ge w+1}\cF)) \to H^0(\D(\cF)) \to H^0(\D(\sigma_{\le w}\cF)) \to H^1(\D(\sigma_{\ge w+1}\cF)) \to \cdots,
\]
when restricted to $U$, gives to the short exact sequence
\[
0 \to \D(\sigma_{\ge w+1}\cF)|_U \to \D(\cF)|_U \to \D(\sigma_{\le w}\cF)|_U \to 0.
\]
Now, by the inductive assumption, we can replace $U$ by a smaller open subscheme so that the first term lies in $\cgl U{\alt X - w -1}$, and by Lemma~\ref{lem:st-dual}, we can similarly assume that the last term lies in $\cgl U{\alt X - w}$.  Since $\cgl U{\alt X-w}$ is stable under extensions, we have $\D(\cF)|_U \in \cgl U{\alt X - w}$, as desired.
\end{proof}

\section{The Staggered $t$-structure: Orthogonality}
\label{sect:stag-orth}

In this section, we define the subcategories of $\dgb X$ that will constitute the staggered $t$-structure, and we establish a number of their basic properties.  The proof that they actually give a $t$-structure will be completed in the next section.

Let $X^\gen$ be the set of generic points of ($G$-invariant) subschemes of
$X$.  For each $x \in X^\gen$, let $\barGx$ denote the smallest ($G$-invariant) closed subset of $X$ containing $x$.  ($\barGx$ is not endowed with a fixed subscheme structure.)
We continue to assume, as in the previous section, that $X$ is endowed with an $s$-structure.

Let us fix, once and for all, an equivariant dualizing complex $\omega_X$ on $X$.  For each $x \in X^\gen$, let $\cod \barGx$ be the degree in which the complex $R\Gamma_x(\omega_X)$ is concentrated.  This is a ``codimension function'' in the sense of~\cite[\S V.7]{har}: for any other point $y \in X^\gen$, we have
\[
\cod \barGx - \cod \barGy = \dim \barGy - \dim \barGx.
\]
(In~\cite{bez:pc}, $\omega_X$ was assumed to be normalized so that $\cod \barGx = -\dim \barGx$, but we do not make that assumption.)  More generally, for any closed subscheme $Z \subset X$, we define
\[
\cod Z = \min \{ \cod \barGx \mid x \in Z^\gen \}.
\]

If $j: U \hto X$ is an open subscheme, then $j^*\omega_X$ is a dualizing complex on $U$, and if $i: Z \hto X$ is a closed subscheme, then $Ri^\flat \omega_X$ is a dualizing complex on $Z$.  Thus, the choice of $\omega_X$ determines a choice of dualizing complex on all locally closed subschemes of $X$.  The notation $\D$, when applied to complexes of sheaves on a subscheme of $X$, should be understood to be defined with respect to the dualizing complex obtained from $\omega_X$ in this way.

\begin{defn}
The \emph{staggered codimension} of an irreducible closed
$G$-invariant subscheme $Y \subset X$, denoted $\scod Y$, is defined to be
$\cod Y + \alt Y$.
\end{defn}

\begin{defn}\label{defn:perv}
A \emph{perversity} is a function $p: X^\gen \to \Z$ such that for any $x,
y \in X^\gen$ with $x \in \barGy$, we have
\[
\begin{aligned}
&\text{(monotonicity)}\qquad & p(x) &\ge p(y) \\
&\text{(comonotonicity)}\qquad& \scod \barGx - p(x) &\ge
 \scod \barGy - p(y)
\end{aligned}
\]
If $p$ is a perversity, then the function $\bar p: X^\gen \to \Z$ defined
by
\[
\bar p(x) = \scod \barGx - p(x)
\]
is clearly also a perversity.  It is called the \emph{dual perversity} to
$p$.
\end{defn}
Note that if $y$ is another generic point of $\barGx$, then the monotonicity condition implies that $p(x) = p(y)$.

Define two full subcategories of $\dgm X$ as follows:
\begin{align*}
\dgml Xn_\st &= \{ \cF \in \dgm X \mid \text{for all $k \in \Z$, $H^k(\cF)
\in \cgl X{n-k}$} \} \\
\dgml Xn_\bt &= \left\{ \cF \in \dgm X \,\bigg|
\begin{array}{c}
\text{for all $k \in \Z$, there is a dense open subscheme} \\
\text{$V \Subset X$ such that $H^k(\cF)|_V \in \cgl V{n-k}$}
\end{array}
\right\}
\end{align*}
Since the various $\cgl X{n-k}$ are stable under extensions, it is clear that $\dgml Xn_\st$ and $\dgml Xn_\bt$ are stable under extensions as well.  It is also clear that both of these categories are stable under all truncation functors (with respect to the standard $t$-structure), since they are defined by conditions on cohomology sheaves.

\begin{rmk}\label{rmk:stag-deg}
The notion of ``staggered degree'' that was informally introduced in Section~\ref{sect:intro} can be formalized as follows: an object $\cF \in \dgb X$ is said to be \emph{concentrated in staggered degree $n$} if for all $k \in \Z$, the sheaf $H^k(\cF)$ is pure of step $n - k$.  This notion is not useful for the general developments of this section or the following one, but it will come up when we study simple objects in Section~\ref{sect:ic}.
\end{rmk}

Given a perversity function $p$, define a full subcategory of $\dgm X$ by
\begin{equation}\label{eqn:stagl}
\p \dgml X0 = \left\{ \cF \,\Bigg| 
\begin{array}{c}
\text{for any $x \in X^\gen$ and any closed} \\
\text{subscheme structure $i: Z\hto X$} \\
\text{on $\barGx$, we have $Li^*\cF \in \dgml Z{p(x)}_\bt$}
\end{array} \right\}.
\end{equation}
Next, we define a full subcategory of $\dgp X$ by
\begin{equation}\label{eqn:stagr}
\p \dgpg X0 = \D(\barp \dgml X0).
\end{equation}
The fact that the various $\dgml Z{p(x)}_\bt$ are stable under extensions implies that $\p \dgml X0$ and $\p \dgpg X0$ are as well.  We also define
\[
\p \dgml Xn = \p \dgml X0[-n]
\qquad\text{and}\qquad
\p \dgpg Xn = \p \dgpg X0[-n],
\]
as well as
\[
\p\dgbl Xn = \p\dgml Xn \cap \dgb X
\qquad\text{and}\qquad
\p\dgbg Xn = \p\dgpg Xn \cap \dgb X.
\]

The rest of this section and the following one will be devoted to the proof of the following theorem, which is one of the main results of the paper.

\begin{thm}\label{thm:stag}
$(\p\dgbl X0, \p\dgbg X0)$ is a nondegenerate,
bounded $t$-structure on $\dgb X$.
\end{thm}

\begin{defn}
The $t$-structure $(\p\dgbl X0, \p\dgbg X0)$ is called the \emph{staggered $t$-structure} on $\dgb X$.  Its heart, denoted $\p\sm X$, is the category of \emph{staggered sheaves}.
\end{defn}

Let $Y \subset X$ be a locally closed subscheme.  Any perversity $p: X^\gen \to \Z$ determines, by restriction, a perversity $Y^\gen \to \Z$ as well.  By an abuse of notation, we denote the latter by $p$ as well.  Then, the categories $\p\dgml Y0$ and $\p\dgpg Y0$ are defined just as in~\eqref{eqn:stagl}--\eqref{eqn:stagr}.

\begin{lem}\label{lem:perv-res}
\begin{enumerate}
\item Let $j: U \hto X$ be an open subscheme.  The functor $j^*$ takes $\p\dgml X0$ to $\p\dgml U0$, and $\p\dgpg X0$ to $\p\dgpg U0$.
\item Let $i: Z \hto X$ be a closed subscheme.  Then $Li^*$ takes $\p\dgml X0$ to $\p\dgml Z0$, and $Ri^\flat$ takes $\p\dgpg X0$ to $\p\dgpg Z0$.
\item Let $i: Z \hto X$ be a closed subscheme.  Then $i_*$ takes $\dgml Zn_\st$ to $\dgml Xn_\st$.
\end{enumerate}
\end{lem}
Corollary~\ref{cor:perv-closed} is another statement that is similar in spirit to this lemma, but we will be able to prove after establishing Theorem~\ref{thm:stag}.
\begin{proof}
(1)~Since $j^*$ is exact and $s$-exact, it is clear that it takes $\dgml Xn_\st$ to $\dgml Un_\st$.  It follows that $j^*$ takes $\p\dgml X0$ to $\p\dgml U0$.  Because $\D$ commutes with $j^*$, we see that $j^*$ takes $\p\dgpg X0$ to $\p\dgpg U0$ as well.

(2)~It is immediate from the definition of $\p\dgml X0$ that $Li^*$ takes $\p\dgml X0$ to $\p\dgml Z0$.  We obtain the statement on $Ri^\flat$ from the fact that $\D \circ Li^* \simeq Ri^\flat \circ \D$.

(3)~This follows from the fact that $i_*$ is exact and $s$-exact.
\end{proof}

The following lemma is an immediate consequence of
Lemma~\ref{lem:omega-conc}.

\begin{lem}
For any $x \in X^\gen$ and any closed subscheme structure $i: Z \hto X$ on $\barGx$,
there is an open subscheme $V \subset Z$ such that $Ri^\flat\omega_X|_V$ is
concentrated in degree $\cod \barGx$.\qed
\end{lem}

\begin{lem}\label{lem:pl-trunc}
If $\cF \in \p \dgml X0$, then for any $n \in \Z$, we have $\tau^{\ge n}\cF \in \p \dgbl X0$ and $\tau^{\le n}\cF \in \p \dgml X0$.
\end{lem}
This really is a one-sided statement: in general, $\p \dgpg X0$ is not stable under standard truncation functors.
\begin{proof}
Let $a$ be the largest integer such that $H^a(\cF) \ne 0$, and form the distinguished triangle
\begin{equation}\label{eqn:pl-trunc}
\tau^{\le a-1}\cF \to \cF \to \tau^{\ge a}\cF \to.
\end{equation}
Let $\cF_0 = H^a(\cF)$, so that $\tau^{\ge a}\cF \simeq \cF_0[-a]$.
For any point $x \in X^\gen$ and any subscheme structure $i: Z \hto X$ on $\barGx$, this gives rise to a distinguished triangle
\[
Li^*\tau^{\le a-1}\cF \to Li^*\cF \to Li^*\cF_0[-a] \to.
\]
Since the first term belongs to $\dgml Z{a-1}$, we see from the long exact cohomology sequence that $H^a(Li^*\cF) \simeq H^a(Li^*\cF_0[-a]) \simeq H^0(Li^*\cF_0) \simeq i^*\cF_0$.  Next, from the definition of $\p \dgml X0$, we know that there is an open subscheme $V \Subset Z$ such that $i^*\cF_0|_V \in \cgl V{p(x)-a}$.  Let $U \Subset X$ be any dense open subscheme such that $U \cap Z = V$.  Then $\cF_0|_U \in \cloc UV{p(x)-a}$.  By Lemma~\ref{lem:li}, for any $k \ge 0$, we can find another open set $V' \Subset V$ such that $H^{-k}(Li^*\cF_0)|_{V'} \in \cgl {V'}{p(x)-a} \subset \cgl {V'}{p(x)+k-a}$, or, equivalently, $H^{a-k}(Li^*\cF_0[-a])|_{V'} \in \cgl {V'}{p(x)+k-a}$.  We have thus shown that $\cF_0[-a] \in \p \dgbl X0$.

Next, for any $k \ge 1$, consider the following exact sequence:
\[
H^{a-k-1}(Li^*\cF_0[-a]) \to H^{a-k}(Li^*\tau^{\le a-1}\cF) \to H^{a-k}(Li^*\cF).
\]
Upon restriction to a suitable open subscheme $V \Subset Z$, we know that the last term yields an object in $\cgl V{p(x)+k-a}$, and the first term yields an object of $\cgl V{p(x)-a} \subset \cgl V{p(x)+k-a}$.  Therefore, we have that $H^{a-k}(Li^*\tau^{\le a-1}\cF)|_V \in \cgl V{p(x)+k-a}$ as well.  This shows that that $\tau^{\le a-1}\cF \in \p \dgml X0$.

We have seen that all three terms of~\eqref{eqn:pl-trunc} lie in $\p \dgml X0$.  Applying this result to $\tau^{\le a-1}\cF$ yields in particular that $\tau^{\ge a-1}(\tau^{\le a-1}\cF) \simeq H^{a-1}(\cF)[1-a] \in \p \dgml X0$.  Indeed, by induction, one finds that for all $k$, we have $\tau^{\le k}\cF \in \p \dgml X0$ and $H^k(\cF)[-k] \in \p \dgbl X0$.  By a further induction argument, using the fact that $\p \dgbl X0$ is stable extensions, the latter fact implies that $\tau^{\ge k}\cF \in \p \dgbl X0$, as desired.
\end{proof}

\begin{prop}\label{prop:stag-hom}
Suppose $\cF \in \p \dgml X0$ and $\cG \in \p \dgpg X1$.  Then we have $\Hom(\cF,\cG) = 0$.
\end{prop}
\begin{proof}
Suppose $\cG \in \dgpg Xa$.  Then we have $\Hom(\cF,\cG) \simeq
\Hom(\tau^{\ge a}\cF,\cG)$, and by the preceding lemma, $\tau^{\ge a}\cF \in \p\dgbl X0$.  By replacing $\cF$ by $\tau^{\ge a}\cF$ if
necessary, we may assume without loss of generality that $\cF$ is actually
a bounded complex.

By definition, there is an object $\cG_1 \in \barp \dgml X{-1}$ such
that $\D(\cG_1) \simeq \cG$.  Now, suppose $\cF \in \dgml Xb$.  Then, as
above, we have $\Hom(\cF, \cG) \simeq \Hom(\cF, \tau^{\le b}\cG)$.  
Consider applying the functor $\D$ to the distinguished
triangle
\[
\tau^{\le \cod X -b-1}\cG_1 \to \cG_1 \to \tau^{\ge \cod X - b}\cG_1 \to.
\]
Recall that $\omega_X \in \dgpg X{\cod
X}$.  We therefore have
\[
\D(\tau^{\le \cod X -b-1}\cG_1) \simeq \cRHom(\tau^{\le \cod X
-b-1}\cG_1, \omega_X) \in \dgpg X{b+1}.
\]
In particular, it follows $\tau^{\le b}\D(\tau^{\ge \cod X - b}\cG_1)
\simeq \tau^{\le b}\D(\cG_1) \simeq \tau^{\le b}\cG$.  That is, we may
replace $\cG_1$ by $\tau^{\ge \cod X - b}\cG_1$ without affecting the
calculation of $\Hom(\cF,\cG)$.  Thus, we henceforth assume that $\cG_1$
is a bounded complex as well.  

Choose a generic point $x \in X$, and let $U \subset X$ be an irreducible open subscheme containing $x$ such that
\begin{equation}\label{eqn:stag-hom-u}
\cF|_U \in \dgml U{p(x)}_\bt
\qquad\text{and}\qquad
\cG_1|_U \in \dgml U{\bar p(x)-1}_\bt.
\end{equation}
We will first show by induction on the number of nonzero cohomology sheaves of $\cF$ and $\cG_1$ that, possibly after replacing $U$ by a smaller open subscheme, we may assume that $\Hom(\cF|_U,\cG|_U) = 0$.

Assume for the moment that $\cF$ and $\cG_1$ are each actually concentrated in a single degree, say in degrees $a$ and $b$, respectively.  Then $\cF
\simeq \cF'[-a]$ and $\cG_1 \simeq \cG'_1[-b]$ for some sheaves $\cF', \cG'_1
\in \cg X$.  Replacing $U$ be a smaller open subscheme if necessary, we may assume that $H^a(\cF)|_U \simeq \cF'|_U \in \cgl U{p(x)-a}$ and $H^b(\cG)|_U \simeq \cG'_1|_U \in \cgl U{\bar p(x)-1-b}$.  By Lemma~\ref{lem:st-dual}, we may further assume, after perhaps replacing $U$ by a yet smaller subscheme, that $\D(\cG_1)|_U$ is concentrated in a single degree.  That degree must be $\cod \barGx -b$.  Thus, on $U$, we
have
\[
\D(\cG_1)|_U \simeq \cRHom(\cG'_1[-b],\omega_X|_U) \simeq \cHom(\cG'_1|_U,
\omega_X[\cod \barGx]|_U)[-\cod \barGx + b].
\]
Now,
\begin{align*}
\Hom(\cF|_U, \cG|_U) &\simeq \Hom(\cF'[-a]|_U, \cD(\cG_1)|_U) \\
&\simeq \Hom(\cF'[-a]|_U, \cHom(\cG'_1|_U, \omega_X[\cod \barGx]|_U)[-\cod \barGx +b]) \\
&\simeq \Ext^{-\cod X +a+b}(\cF'|_U, \cHom(\cG'_1|_U, \omega_X[\cod \barGx]|_U)).
\end{align*}
If $-\cod \barGx +a+b < 0$, then we evidently have $\Hom(\cF|_U,\cG|_U)
= 0$.  Assume instead now that $-\cod \barGx +a+b \ge 0$.  Next, by replacing
$U$ by a smaller open subscheme if necessary, we may assume that
$\omega_X[\cod \barGx]|_U \in \cgg U{\alt \barGx}$ (see
Definition~\ref{defn:alt}).  Since $\cG'_1|_U \in \cgl U{\bar p(x)-1-b}$,
we see that
\begin{multline*}
\cHom(\cG'_1|_U, \omega_X[\cod \barGx]|_U) \in \cgg U{\alt \barGx - \bar p(x) +1+b}
\\
= \cgg U{-\cod \barGx + p(x) +b+1} \subset \cgg U{-a + p(x)+1},
\end{multline*}
where the last inclusion comes from the fact that $-\cod \barGx +b \ge -a$. 
Since $\cF'|_U \in \cgl U{p(x)-a}$,  we may invoke axiom~\sref{ax:ext2}:
after replacing $U$ by a smaller open subscheme, we may assume that 
\[
\Ext^{-\cod \barGx +a+b}(\cF'|_U, \cHom(\cG'_1|_U, \omega_X[\cod \barGx]|_U)) = 0,
\]
and hence that $\Hom(\cF|_U, \cG|_U) = 0$, as desired.

Now, we return to the case where $\cF$ is a bounded complex, not necessarily concentrated in a single degree.  $\cG_1$ is still assumed to be concentrated in degree $b$.  Choose $U$ afresh as we originally did: it should be an irreducible open subscheme containing $x$ and satisfying the conditions~\eqref{eqn:stag-hom-u}.  Any distinguished triangle of the form
\[
\tau^{\le k}\cF|_U \to \cF|_U \to \tau^{\ge k+1}\cF|_U \to
\]
gives rise to a sequence
\[
\Hom(\tau^{\ge k+1}\cF|_U, \cG|_U) \to \Hom(\cF|_U,\cG|_U) \to \Hom(\tau^{\le k}\cF|_U,
\cG|_U)
\]
that is exact at the middle term.  The categories $\dgml U{p(x)}_\bt$ and
$\dgml U{\bar p(x)-1}_\bt$ are stable under all truncation functors: we have 
$\tau^{\ge k+1}\cF|_U \in \dgml U{p(x)}_\bt$ and $\tau^{\le k}\cF|_U \in \dgml U{p(x)}_\bt$.  By choosing $k$ such that $\tau^{\ge k+1}\cF$ and $\tau^{\le k}\cF$ each have fewer nonzero cohomology sheaves than $\cF$, we may assume inductively that $U$ has been chosen so that the first and last $\Hom$-groups above vanish.  Obviously, this implies that $\Hom(\cF|_U, \cG|_U) = 0$, as desired.

An analogous argument allows us to generalize to the case where $\cG_1$ is a bounded complex not necessarily in a single degree.

Thus, in all cases, we have now constructed an open subscheme $U \subset
X$ such that $\Hom(\cF|_U,\cG|_U) = 0$.  Let $Z$ be a closed
subscheme complementary to $U$, and consider the exact sequence ({\it cf.} Lemma~\ref{lem:found}\eqref{it:f-les})
\[
\lim_{\substack{\to \\ Z'}} \Hom(Li^*_{Z'}\cF, Ri^\flat_{Z'}\cG) \to
\Hom(\cF,\cG) \to \Hom(\cF|_U,\cG|_U).
\]
We now know that the last term vanishes.  The first term vanishes by the
noetherian induction assumption, using Lemma~\ref{lem:perv-res}.  We
conclude that $\Hom(\cF,\cG) = 0$, as desired.
\end{proof}

\section{The Staggered $t$-structure: Distinguished Triangles}
\label{sect:stag-dt}

We will complete the proof of Theorem~\ref{thm:stag} in this section.  We retain all the notation and assumptions of the preceding section.

\begin{prop}\label{prop:st-perv}
Suppose $\cF \in \dzsupp \cap \dgml Xn_\st$.  If $n \le p(x)$ for all $x \in Z^\gen$, then $\cF \in \p\dgml X0$.
\end{prop}
\begin{proof}
Let $x \in X^\gen$, and let $\kappa: Y \hto X$ be a closed subscheme structure on $\barGx$.  We must show that $L\kappa^*\cF \in \dgml Y{p(x)}_\bt$.  If $x \notin Z$, this is trivially true: then $Y$ contains an open subscheme $V$ not meeting $Z$, necessarily dense since $Y$ is irreducible, and $H^k(L\kappa^*\cF)|_V = 0$ for all $k$.  Assume now that $x \in Z$.  Since $\cF$ is a bounded complex and $\p\dgml X0$ is stable under extensions, we can use distinguished triangles of the form
\[
\tau^{\le k}\cF \to \cF \to \tau^{\ge k+1}\cF \to
\]
and induction on the number of nonzero cohomology sheaves of $\cF$ to reduce to the case where $\cF$ is concentrated in a single degree.  (Here, we also use the fact that $\dgml Xn_\st$ and $\dzsupp$ are stable under truncation functors.)

Suppose, then, that $\cF$ is concentrated in degree $a$: we have $\cF \simeq \cF_1[-a]$ for some sheaf $\cF_1 \in \cgl X{n-a}$.  By Lemma~\ref{lem:li}, for all $k \le a$, there is an open subscheme $V$ such that
\[
H^k(L\kappa^*\cF)|_V \simeq H^{k-a}(L\kappa^*\cF_1)|_V \in \cgl V{n-a} \subset \cgl V{p(x) - k},
\]
where we have made use of the assumption that $n \le p(x)$.  Of course, $H^k(L\kappa^*\cF) = 0$ if $k > a$.  Thus, we see that $L\kappa^*\cF \in \dgml Y{p(x)}_\bt$, as desired.
\end{proof}

\begin{lem}\label{lem:stag-dt1}
Let $U \subset X$ be an open subscheme, and let $Z \subset X$ be the
complementary closed subscheme.  Let $x$ be a generic point of $Z$, and let $\cF \in \dgb X^{[a,b]}$ be such that $\cF|_U \in \dgbl Un_\st$.  Then there is an open subscheme $V \subset X$ with $U \subset V$ and $x \in V$, and a distinguished triangle
\[
\cF' \to \cF \to \cF'' \to
\]
with $\cF' \in \dgb X^{[a,b]}$, $\cF'|_V \in \dgbl Vn_\st$, and $\cF'' \in
\dzsupp$.  In addition, we have $\D(\cF'')|_V \in \dgbl V{d-n-1}_\st$, where $d = \scod \barGx$. 
\end{lem}
\begin{proof}
We proceed by induction on $b - a$.  (The base case $a = b$ will be treated
as part of the general case.)  Form the distinguished triangle
\[
\tau^{\le a}\cF \to \cF \to \tau^{\ge a+1}\cF \to.
\]
Note that $\tau^{\le a}\cF|_U$ and $\tau^{\ge a+1}\cF|_U$ both belong to $\dgbl Un_\st$, since that category is stable under truncation.  Evidently, $\tau^{\le a}\cF \in \dgb X^{[a,a]}$.  Let $\cF_1 = H^a(\cF) \simeq (\tau^{\le
a}\cF)[a]$, and form the short
exact sequence 
\[
0 \to \sigma_{\le n-a}\cF_1 \to \cF_1 \to \sigma_{\ge n-a+1}\cF_1 \to
0.
\]
Now, the definition of $\dgbl Un_\st$ implies that $\cF_1|_U \in \cgl U{n-a}$.  It follows
that $(\sigma_{\ge n-a+1}\cF_1)|_U = 0$, or in other words, that
$\sigma_{\ge n-a+1}\cF_1$ is supported on $\uZ$.  Therefore, there is some subscheme structure $i_{Z'}: Z' \hto X$ on $\uZ$ and some sheaf $\cG_1 \in \cgg{Z'}{n-a+1}$ such that $\sigma_{\ge n-a+1}\cF_1 \simeq i_{Z'*}\cG_1$.  Recall  that $x$ is a generic point of $Z$.  There certainly exist open irreducible subschemes $V'_1 \subset Z'$ containing $x$.  By Proposition~\ref{prop:st-dual}, we can choose a $V'_1$ with the property that $\D(\cG_1)|_{V'_1}$ is concentrated in degree $\cod \barGx$, and such that the sheaf $\D(\cG_1)[\cod \barGx]|_{V'_1}$ is in $\cgl {V'_1}{\alt \barGx - (n-a+1)}$.

Now, let $V' =  X \ssm (Z' \ssm V'_1)$.  This is an open subscheme of $X$ containing $U$, and containing $V'_1$ as a closed subscheme.  Let $\cG' = \sigma_{\le n-a}\cF_1[-a]$, and let $\cG'' = \sigma_{\ge n-a+1}\cF_1[-a] \simeq i_{Z'*}\cG_1[-a]$.  The
short exact sequence above gives rise to a distinguished triangle 
\begin{equation}\label{eqn:stag1}
\cG' \to \tau^{\le a}\cF \to \cG'' \to,
\end{equation}
with $\cG' \in \dgb X^{[a,a]}$, $\cG'|_{V'} \in \dgbl {V'}n_\st$ and $\cG'' \in
\dzsupp$.  Moreover, the calculations above with $\cG_1$ imply that $\D(\cG'')|_{V'}$ is concentrated in degree $\cod \barGx - a$, and that
\begin{multline*}
H^{\cod \barGx - a}(\D(\cG'')|_{V'}) \in \cgl V{\alt \barGx - (n-a+1)} \\
= \cgl V{(\alt\barGx + \cod\barGx -n -1)-(\cod\barGx - a)}.
\end{multline*}
Thus, $\D(\cG'')|_{V'} \in \dgbl {V'}{d-n-1}_\st$, where $d = \scod
\barGx$.
If $b = a$, then $\cF \simeq \tau^{\le a}\cF$, and we are finished: the distinguished triangle in~\eqref{eqn:stag1} is the one we seek.

Now, suppose $b > a$.  Then $\tau^{\ge a+1}\cF \in \dgb X^{[a+1,b]}$. By 
the inductive assumption, we can find an open subscheme $V'' \subset X$
with $U \subset V''$ and $x \in V'' \cap Z$, and a distinguished triangle
\[
\cH' \to \tau^{\ge a+1}\cF \to \cH'' \to
\]
with $\cH' \in \dgb X^{[a+1,b]}$, $\cH'|_{V''} \in \dgbl{V''}n_\st$, $\cH'' \in
\dzsupp$, and $\D(\cH'')|_{V''} \in \dgbl{V''}{d-n-1}_\st$.  Let $W = V'
\cap V''$.  This is an open subscheme with $U \subset W$ and $x \in W$.  Note that
\[
\cG''[a]|_{W} \simeq i_*\cG_1|_{W} \in \csupp{W}{W \cap V'_1}_{\ge n-a+1}.
\]
On the other hand, we also have $\cH'|_W \in \dgbl Wn_\st$, so whenever $a < k \le b$, we have $H^k(\cH')|_W \in \cgl W{n-k} \subset \cgl W{n-a}$.  Thus, by invoking axiom~\sref{ax:ext2} finitely many times, we may find a subscheme $V \Subset_{W \cap V'_1} W$ such that
\[
\gExt^{1+k-a}(H^k(\cH')|_V, \cG''[a]|_V) = 0
\]
for all $k$ such that $a < k \le b$.  Note that $V \cap Z$ must contain $x$.

We claim that $\Hom(\cH'|_V, \cG''[1]|_V) = 0$.  Indeed, we will show by
induction on $k$ that $\Hom(\tau^{\le k}\cH'|_V, \cG''|_V) = 0$ for all $k$ with $a \le k \le b$.  It is trivially true for $k = a$, since $\tau^{\le a}\cH' = 0$. 
Now, if $k \ge a+1$, form the distinguished triangle 
\[
\tau^{\le k-1}\cH' \to \tau^{\le k}\cH' \to \tau^{[k,k]}\cH' \to.
\]
This gives rise to a sequence
\[
\Hom(\tau^{[k,k]}\cH'|_V, \cG''[1]|_V) \to \Hom(\tau^{\le k}\cH'|_V,
\cG''[1]|_V) \to  \Hom(\tau^{\le k-1}\cH'|_V,\cG''[1]|_V).
\] 
Here, the last term vanishes by assumption.  On the other hand, we have
$\tau^{[k,k]}\cH' \simeq H^k(\cH')[-k]$, and
\[
\Hom(H^k(\cH')[-k]|_V, \cG''[1]|_V) \simeq \Ext^{1+k-a}(H^k(\cH')|_V,
\cG''[a]|_V) = 0. 
\]
We now see from the sequence above that $\Hom(\tau^{\le k}\cH'|_V,
\cG''[1]|_V) = 0$.  Since $\tau^{\le b}\cH' \simeq \cH'$, we
see that $\Hom(\cH'|_V, \cG''[1]|_V) = 0$, as desired. 

Next, consider the sequence
\[
\Hom(\cH'|_V, \cG'[1]|_V) \to \Hom(\cH'|_V, (\tau^{\le a}\cF)[1]|_V) \to
\Hom(\cH'|_V,\cG''[1]|_V).
\]
Since the last term vanishes, we see that every morphism $\cH'|_V \to
(\tau^{\le a}\cF)[1]|_V$ must factor through the map $\cG'[1]|_V \to
(\tau^{\le a}\cF)[1]|_V$.  In particular, consider the composition 
\[
\cH'|_V \to \tau^{\ge a+1}\cF|_V \to (\tau^{\le a}\cF)[1]|_V.
\]
This gives rise to a commutative square
\[
\xymatrix@=10pt{
\cH'|_V \ar[r]\ar[d] & \tau^{\ge a+1}\cF|_V \ar[d] \\
\cG'[1]|_V \ar[r] & (\tau^{\le a}\cF)[1]|_V}
\]
Let us complete this diagram using the $9$-lemma, and then rearrange it:
\[
\xymatrix@=10pt{
\cG'|_V \ar[r]\ar[d] & \tau^{\le a}\cF|_V \ar[r]\ar[d]
& \cG''|_V \ar[r]\ar[d] &\\
\cF'_V \ar[r]\ar[d] & \cF|_V \ar[r]\ar[d] & \cF''_V \ar[r]\ar[d] &\\
\cH'|_V \ar[r]\ar[d] & \tau^{\ge a+1}\cF|_V \ar[r]\ar[d] & \cH''|_V \ar[r]\ar[d] &\\
&&&}
\]
Recall that all rows and columns in this diagram are distinguished
triangles.  The categories $\dgb V^{[a,b]}$ and $\dgbl Vn_\st$ are stable
under extensions, so we clearly have $\cF'_V \in \dgbl Vn_\st \cap
\dgb V^{[a,b]}$.  It is also evident that $\cF''_V \in \dsupp{V}{V \cap
Z}$.  Since $\dgbl V{d-n-1}_\st$ is stable under extensions, we have $\D(\cF''_V) \in \dgbl V{d-n-1}$ as well.  Now, by Lemma~\ref{lem:subcplx-extend}, we
can find a distinguished triangle
\[
\cF' \to \cF \to \cF'' \to
\]
with $\cF' \in \dgb X^{[a,b]}$, $\cF'|_V \simeq \cF'_V$, and $\cF''|_V
\simeq \cF''_V$, and such that the restriction to $V$ of this distinguished
triangle is isomorphic to the middle row of the diagram obtained from the
$9$-lemma above.  This distinguished triangle has all the properties we seek.
\end{proof}

Below, we will make use of the ``$*$''
operation on a triangulated category introduced in~\cite[\S 1.3]{bbd}. 
Recall that if $\cA_1$ and $\cA_2$ are classes of objects in a triangulated
category, then
\[
\cA_1 * \cA_2
\]
is the class of all objects $A$ that can be embedded in a distinguished
triangle
\[
A_1 \to A \to A_2 \to
\]
with $A_1 \in \cA_1$ and $A_2 \in \cA_2$.  According
to~\cite[Lemme~1.3.10]{bbd}, the $*$ operation is associative, so
expressions like $\cA_1 * \cA_2 * \cA_3$ are well-defined.  This fact will
be used freely below.

\begin{lem}\label{lem:stag-dt2}
Let $U \subset X$ be an open subscheme, and let $Z \subset X$ be the
complementary closed subscheme.  Let $\cF \in \dgb X^{[a,b]}$ be such that $\cF|_U \in \dgbl Un_\st$.  Then $\cF \in \dgbl Xn_\st * \dzsupp$.
\end{lem}
\begin{proof}
This lemma follows by noetherian induction from Lemma~\ref{lem:stag-dt1}. 
Specifically, let us assume that the lemma is known if $Z$ is replaced by
any proper closed subscheme.  (In the base case, where $U = X$ and $Z =
\varnothing$, the
lemma is trivial.)  Let $\cF \in \dgb X^{[a,b]}$ with $\cF|_U \in \dgbl Un_\st$, and choose a generic point $x \in Z$.  Lemma~\ref{lem:stag-dt1} tells us that
\[
\cF \in \{\cF'\} * \dzsupp,
\]
for some $\cF' \in \dgb X^{[a,b]}$ with $\cF'|_V \in \dgbl Vn_\st$, 
where $V \supset U$.  Let $Y \subset Z$ be a closed subscheme structure
on the complement of $V \cap Z$.  Then $Y$ is a proper closed subscheme of $Z$, so we
know by assumption that
\[
\cF' \in \dgbl Xn_\st * \dzsupp.
\]
It is clear that $\dzsupp * \dzsupp = \dzsupp$, and since $*$ is
associative, the lemma follows.
\end{proof}

\begin{prop}\label{prop:stag-t}
We have $\dgb X = \p\dgb X^{\le 0} * \p\dgb X^{\ge 1}$.
\end{prop}
\begin{proof}
Let $\cF \in \dgb X$.  We proceed by noetherian induction on the support of $\cF$.  Suppose $\cF$ is supported on a closed subscheme $\kappa: Y \to X$.  Then there is some $\cF_1 \in \dgb Y$ such that $\cF \simeq \kappa_*\cF_1$.  Choose a generic point $x \in Y^\gen$ such that $p(x) \le p(y)$ for all $y \in Y^\gen$.  (This minimum value of $p$ must be achieved on one of the finitely many generic points of $Y$.)  Now, apply Lemma~\ref{lem:stag-dt1} to $\cF_1$.  (In the notation of that statement, we are taking $U = \varnothing$.)  We obtain an open subscheme $V \subset Y$ containing $x$, and a distinguished triangle
\begin{equation}\label{eqn:t1}
\cF'_1 \to \cF_1 \to \cF''_1 \to
\end{equation}
where $\cF'_1|_V \in \dgbl V{p(x)}_\st$ and $\D(\cF''_1)|_V \in \dgbl
V{\scod \barGx - p(x) - 1}_\st = \dgbl V{\bar p(x) - 1}_\st$.  Next, let $Z
\subset Y$ be a closed subscheme complementary to $V$.  By
Lemma~\ref{lem:stag-dt2}, we know that
\[
\cF'_1 \in \dgbl Y{p(x)}_\st * \dsupp YZ
\qquad\text{and}\qquad
\D(\cF''_1) \in \dgbl Y{\bar p(x)-1}_\st * \dsupp YZ.
\]
From the distinguished triangle~\eqref{eqn:t1}, we see now that
\begin{align*}
\cF_1 \in {}&(\dgbl Y{p(x)}_\st * \dsupp YZ) * \D(\dgbl Y{\bar p(x)-1}_\st
* \dsupp YZ) \\
&= \dgbl Y{p(x)}_\st * \dsupp YZ * \D(\dsupp YZ) * \D(\dgbl Y{\bar p(x)-1}_\st) \\
&= \dgbl Y{p(x)}_\st * \dsupp YZ * \D(\dgbl Y{\bar p(x)-1}_\st).
\end{align*}
Recall from Lemma~\ref{lem:perv-res} that $\kappa_*$ takes $\dgbl Yn_\st$ to $\dgbl Xn_\st$.  This functor also commutes with $\D$, so we find that
\[
\cF \in (\dsupp XY \cap \dgbl X{p(x)}_\st) * \dzsupp * \D(\dsupp XY \cap \dgbl X{\bar p(x)-1}_\st).
\]
Now, since $p(x)$ is the minimum value of $p$ on $Y^\gen$, Proposition~\ref{prop:st-perv} allows us to deduce that
\begin{multline*}
\cF \in \p\dgbl X0 * \dzsupp * \D(\barp \dgbl X{-1}) \\
= \p\dgbl X0 * \dzsupp * \p\dgbg X1.
\end{multline*}
We have assumed inductively that $\dzsupp \subset \p\dgbl X0 * \p\dgbg X1$, so we can conclude that $\cF \in \p\dgbl X0 * \p\dgbg X1$.
\end{proof}

\begin{proof}[Proof of Theorem~\ref{thm:stag}]
It is obvious from the definitions that
\[
\p\dgbl X{-1} \subset \p\dgbl X0
\qquad\text{and}\qquad
\p\dgbg X0 \subset \p\dgbg X1. 
\]
The other axioms for a $t$-structure have
been checked in Propositions~\ref{prop:stag-hom} and~\ref{prop:stag-t}.  Thus, the
categories $(\p\dgbl X0, \p\dgbg X0)$ define a $t$-structure on $\dgb X$.

Next, we show that it is nondegenerate.  We first prove that there is no
nonzero object belonging to $\p\dgml Xn$ for all $n \in \Z$.  We proceed by
noetherian induction: assume that for any proper closed subscheme $Z \subset X$,
we have $\bigcap_{n \in \Z} \p\dgml Zn = \{0\}$.  The base case is that in
which $X$ contains no nonempty proper closed
($G$-invariant) subschemes.  In that case, we simply have $\p\dgml X0 =
\dgml X{p(x)}_\bt = \dgml X{p(x)}_\st$, where $x$ is any generic point of
$X$.  Now, if we had an object $\cF$ that was in $\p\dgml Xn$ for all $n
\in \Z$, that would imply that for every $k \in \Z$, $H^k(\cF) \in \cgl
X{p(x)+n-k}$ for all $n \in \Z$.  Using axiom~\sref{ax:bdd}, we see
that that implies that $H^k(\cF) = 0$ for all $k$, so in fact $\cF = 0$, as
desired.

In the general case, let $\cF \in \dgm X$ be a nonzero object.  Then there certainly exists some proper closed subscheme $i: Z \hto X$ such that $Li^*\cF \in \dgm Z$ is nonzero.  If $\cF \in \p\dgml Xn$ for all $n$, then by Lemma~\ref{lem:perv-res}, $Li^*\cF \in \p\dgml Zn$ for all $n$, but that is a contradiction.  We conclude that $\bigcap_{n \in \Z} \p\dgml Xn = \{0\}$.

Applying $\D$ lets us conclude that there is no nonzero object belonging to all $\p\dgpg Xn$ either.  Thus, the $t$-structure on $X$ is nondegenerate.

Finally, we show that it is bounded.  Any $\cF \in \dgb X$ has finitely
many nonzero cohomology sheaves, of course.  For each $k \in \Z$ such that
$H^k(\cF) \ne 0$, let $v_k \in \Z$ be such that $H^k(\cF) \in \cgl X{v_k}$.
(Such a $v_k$ exists by axiom~\sref{ax:bdd}.)  Then, let $d$ be the
minimal value of the perversity $p$, and define
\[
n = \max \{v_k + k \mid H^k(\cF) \ne 0\} - d.
\]
We then have $v_k \le n+d-k$ for each $k$, so we see that $H^k(\cF) \in
\cgl X{n+d-k}$.  Thus, $\cF \in \dgbl X{n+d}_\st$, or equivalently, $\cF[n]
\in \dgbl Xd_\st$.  Propostion~\ref{prop:st-perv} now tells us that $\cF[n]
\in \p\dgbl X0$, or $\cF \in \p\dgbl Xn$.  A similar argument allows to
find an integer $m$ such that $\D(\cF) \in \barp\dgbl Xm$, and from that we
deduce that $\cF \in \p\dgbg X{-m}$.  Thus, the staggered $t$-structure is
bounded.
\end{proof}

\begin{cor}\label{cor:perv-closed}
Let $i: Z \hto X$ be a closed subscheme.  Then $i_*$ takes $\p\dgbl Z0$ to $\p\dgbl X0$ and $\p\dgbg Z0$ to $\p\dgbg X0$.
\end{cor}
\begin{proof}.
Let $\cF \in \p\dgbl Z0$.  Since we now know that $(\p\dgbl X0, \p\dgbg X0)$ is a $t$-structure on $\dgb X$, we know that $i_*\cF \in \p\dgbl X0$ if and only if $\Hom(i_*\cF,\cG) = 0$ for all $\cG \in \p\dgbg X1$.  But $\Hom(i_*\cF,\cG) \simeq \Hom(\cF, Ri^\flat\cG)$, and $Ri^\flat\cG \in \dgbg Z1$ by Lemma~\ref{lem:perv-res}.  We see that $\Hom(\cF, Ri^\flat\cG) = 0$ for all $\cG \in \dgbg X1$, so $i_*\cF \in \dgbl X0$.  A similar argument shows that $i_*$ takes $\p\dgbg Z0$ to $\p\dgbg X0$.
\end{proof}

\section{Simple Staggered Sheaves}
\label{sect:ic}

Recall from Section~\ref{sect:intro} that under certain circumstances, simple objects in the category of perverse coherent sheaves have a remarkably easy description (in terms of ``middle-extension functors''), and that all objects in the category have finite length.  In this section, we develop analogues of these results for staggered sheaves.  The proofs follow those given in~\cite[\S 3.2]{bez:pc} for perverse coherent sheaves quite closely, and in many cases we have omitted most details.

In this section and the following one, we assume that the base scheme over which we are working is $\Spec \Bbbk$ for some field $\Bbbk$, and we assume that $G$ is a linear algebraic group over $\Bbbk$.  An \emph{orbit} in $X$ will mean a reduced, locally closed subscheme isomorphic to a homogeneous space for $G$.

Let $j: U \hto X$ be a dense open subscheme, and let $Z \subset X$ be its complement.  Given a perversity $p$, define new functions $p^-, p^+: X^\gen \to \Z$ by
\[
p^-(x) = 
\begin{cases}
p(x) & \text{if $x \in U^\gen$,} \\
p(x)-1 & \text{if $x \notin U^\gen$,}
\end{cases}
\qquad
p^+(x) = 
\begin{cases}
p(x) & \text{if $x \in U^\gen$,} \\
p(x)+1 & \text{if $x \notin U^\gen$.}
\end{cases}
\]
These are not necessarily perversities: they need not, in general, have the monotonicity and comonotonicity properties of Definition~\ref{defn:perv}, so they may not give rise to $t$-structures on $\dgb X$.  Nevertheless, the definitions of the categories ${}^{p^-}\dgml X0$ and ${}^{p^+}\dgpg X0$ still make sense.

\begin{lem}\label{lem:ic-subquot}
Let $\cF \in \p\sm X$.  Then:
\begin{enumerate}
\item $\cF \in {}^{p^-}\dgbl X0$ if and only if $\Hom(\cF,\cG) = 0$ for all $\cG \in \p\sm X \cap \dzsupp$.
\item $\cF \in {}^{p^+}\dgbg X0$ if and only if $\Hom(\cG,\cF) = 0$ for all $\cG \in \p\sm X \cap \dzsupp$.
\end{enumerate}
\end{lem}
\begin{proof}
This follows from Lemma~\ref{lem:perv-res}, Corollary~\ref{cor:perv-closed}, and general properties of $t$-structures, as in the proof of~\cite[Lemma~6]{bez:pc}.
\end{proof}

\begin{prop}\label{prop:ic}
Assume that $p^-$ and $p^+$ are themselves perversities, and define a full subcategory $\p\smpm X \subset \p\sm X$ by $\p\smpm X = {}^{p^-}\dgbl X0 \cap {}^{p^+}\dgbg X0$.  The functor $j^*$ induces an equivalence of categories $\p\smpm X \to \p\sm U$.
\end{prop}
\begin{proof}
The arguments given in either~\cite[Theorem~2]{bez:pc} or~\cite[Proposition~2.3]{as} can be repeated verbatim, with one minor change: references to~\cite[Lemma~6]{bez:pc} should be replaced by references to Lemma~\ref{lem:ic-subquot} above.  We omit further details.
\end{proof}

\begin{rmk}
The assumption that $p^-$ and $p^+$ are themselves perversities has as its main consequence the requirement that for any $x \in U^\gen$ and $y \in Z^\gen$ with $y \in \barGx$, we must have $\scod \barGy \ge \scod \barGx + 2$.  (This follows easily from the monotonicity and comonotonicity requirements and the fact that $p^+(y) = p^-(y)+2$.)  A similar restriction appears in~\cite[\S 3.2]{bez:pc}.

A different approach, introduced in~\cite[Propostions~2.3, 2.6]{as}, allows one to prove a version of this result in which it is not necessary to assume that $p^-$ and $p^+$ are perversities.  (To be more precise, their definitions are changed in such a manner that they are guaranteed to be perversities.)  That approach could certainly be duplicated here.

However, the resulting equivalence of categories has the disadvantage that it relates $\p\smpm X$ to a certain full (but not Serre!) subcategory of $\p\sm U$, not $\p\sm U$ itself.  We will not require the generalized form of the middle-extension functor, and indeed, the generalized version cannot be used to prove an analogue of Theorem~\ref{thm:artin}.  We will therefore content ourselves with the version stated above.
\end{rmk}

\begin{defn}
The inverse equivalence to the one described in Proposition~\ref{prop:ic}, denoted $j_{!*}: \p\sm U \to \p\smpm X$, is called the \emph{middle-extension functor}.
\end{defn}

\begin{defn}
Let $\kappa: Y \hto X$ be a closed subscheme, and let $h: V \hto Y$ be an open subscheme of $Y$.  If $h_{!*}$ is defined, then given any $\cF \in \p\sm V$, we define an object of $\p\sm X$ by
\[
\cIC(Y,\cF) = \kappa_*(h_{!*}\cF).
\]
This is called the (\emph{staggered}) \emph{intersection cohomology complex} associated to $\cF$.
\end{defn}

\begin{prop}\label{prop:simple-orbit}
Suppose $X$ has the property that its associated reduced scheme $X_\red$ is a single $G$-orbit.  Let $d$ be the value of the perversity $p$ at any generic point of $X$.  Then
\[
\p\sm X = \{ \cF \in \dgb X \mid \text{$\cF$ is concentrated in staggered degree $d$} \}.
\]
In addition, let $t: X_\red \to X$ be the inclusion map.  For any irreducible vector bundle $\cL \in \cg {X_\red}$, $t_*\cL[-d+\step\cL]$ is a simple object of $\p\sm X$, and every simple object arises in this way.
\end{prop}
(See Remark~\ref{rmk:stag-deg} for the notion of staggered degree.)
\begin{proof}
Since $X_\red$ is a single $G$-orbit, $X$ obviously has no nonempty proper open subschemes, and no nonempty proper closed subschemes except those whose underlying topological space is all of $X$.  We see, therefore, that $\dgml Xn_\bt = \dgml Xn_\st$.  It follows then from Proposition~\ref{prop:st-perv} that
$\p\dgml X0 = \dgml Xd_\st$.  Now, suppose $\cF \in \p\dgpg X0$, so that $\D(\cF) \in \barp\dgml X0 = \dgml X{\scod X - d}_\st$.  A straightforward induction argument on the number of nonzero cohomology sheaves of $\cF$, combined with Remark~\ref{rmk:homl} and Lemma~\ref{lem:dual-conc}, shows that $H^k(\cF) \in \cgg X{d-k}$ for all $k$.  We conclude, therefore, that $\cF \in \p\sm X$ if and only if $H^k(\cF) \in \cgl X{d-k} \cap \cgg X{d-k}$, as desired.

Under the circumstances of this proposition, we see that $\p\sm X$ is stable under all (standard) truncation functors.  (This is certainly not true in general.)  It is then obvious that a simple object of $\p\sm X$ must be concentrated in a single degree.  Determining the simple objects in $\p\sm X$ then essentially reduces to determining the simple objects in $\cg X$ (up to shifting, to give objects the correct staggered degree).  Of course, every simple object of $\cg X$ is supported on $X_\red$.  Since $X_\red$ is a single $G$-orbit, all objects of $\cg {X_\red}$ are locally free; {\it i.e.}, they are vector bundles on $X_\red$.  Thus, the simple objects of $\p\sm X$ are precisely the objects in $\dgb X$ of the form $t_*\cL[-d+\step \cL]$, with $\cL \in \cg {X_\red}$ an irreducible vector bundle.
\end{proof}

If $C \subset X$ is a single $G$-orbit, we write $p(C)$ for the value of the perversity function $p$ at any generic point of $C$.  They are necessarily all equal, so $p(C)$ is well-defined.

\begin{thm}\label{thm:simple}
Let $C \subset X$ be a $G$-orbit.  Assume that for any point $y \in
\overline C{}^\gen \ssm C^\gen$, we have $p(y) > p(C)$ and $\bar p(y) > \bar
p(C)$. Then, for any object $\cF \in \p\sm X \cap \dsupp X{\overline C}$,
the following conditions are equivalent:
\begin{enumerate}
\item $\cF$ is a simple staggered sheaf.\label{it:simple}
\item $\cF
\simeq \cIC(\overline C, \cL[-p(C)+\step \cL])$, where $\cL
\in \cg C$ is some irreducible vector bundle.\label{it:ic}
\end{enumerate} 
\end{thm}

\begin{rmk}
The analogue of this statement for perverse coherent sheaves, given
in~\cite[Corollary~4]{bez:pc}, is significantly stronger: the conditions on
the perversity are not imposed at the outset, but deduced from the
assumption that $\cF$ is simple, using Nakayama's Lemma; furthermore, the
fact that the support of any simple object is the closure a single
$G$-orbit is obtained from Rosenlicht's Theorem.  This stronger statement
is not true in general for staggered sheaves, in part because there is no
analogue of Nakayama's Lemma for $s$-structures: for a sheaf $\cF \in \cg
X$, knowing that $\cF|_U \notin \cgl Uw$ for some $w$ and some
open subscheme $U$ imposes no restrictions at all on the behavior of $\cF$
with respect to the $s$-structure on the complementary closed subscheme $i:
Z \hto X$.  Similarly, for $\cF \in \p\sm X$, knowing that $\cF|_U \notin
\dgml Un_\st$ for some $n$ does not imply that $Li^*\cF \notin \dgml
Zn_\st$.  Moreover, we cannot invoke Rosenlicht's Theorem, because
the perversity only gives us information about the \emph{staggered}
codimensions of closed subschemes, not their actual codimensions.
\end{rmk}
 
\begin{proof}
We reserve the notation $\overline C$ for the reduced closed subscheme
structure on the closure of $C$.  Let $\cF \in \dsupp X{\overline C}$ be a simple object of $\p\sm X$.  Let $\kappa: Y \hto X$ be a closed subscheme structure on the underlying topological space of $\overline C$ such that $\cF$ is supported on $Y$; suppose $\cF \simeq \kappa_*\cF_1$.  Let $h: V \hto Y$ be the open subscheme with the same underlying topological space as $C$.  It is clear from the hypotheses on the perversity that the middle-extension functor $h_{!*}$ is defined, and then from the construction of $h_{!*}$ in the proof of Proposition~\ref{prop:ic}, we see that $h_{!*}(\cF_1|_V)$ is a subquotient of $\cF_1$.  Since the latter is simple, we see that $\cF_1 \simeq h_{!*}(\cF_1|_V)$, and hence that $\cF \simeq \cIC(Y,\cF_1|_V)$.  Now, because $h_{!*}$ gives an equivalence of categories between $\p\sm V$ and $\p\smpm Y$, we know that $\cF_1|_V$ must be a simple object of $\p\sm V$, and then by Proposition~\ref{prop:simple-orbit}, it is a shift of an irreducible vector bundle on $C$.  Thus, $\cF$ must be of the form $\cIC(\overline C, \cL[-p(C) + \step\cL])$, where $\cL \in \cg C$ is an irreducible vector bundle.

Conversely, if $\cL$ is an irreducible vector bundle on $C$, then it
follows from Lemma~\ref{lem:ic-subquot} that $\cIC(\overline C,
\cL[-p(C)+\step \cL])$ is a simple object in $\p\sm X$.
\end{proof}

\begin{thm}\label{thm:artin}
Suppose $G$ acts on $X$ with finitely many orbits.  In addition, assume
that for any two orbits $C$, $C'$ with $C' \subset \overline C$ and $C' \ne
C$, we have $p(C') > p(C)$ and $\bar p(C') > \bar p(C)$.  Then $\p\sm X$ is a finite-length category.
\end{thm}
\begin{proof}
Identical to~\cite[Corollary~5]{bez:pc}.
\end{proof}

A perversity satisfying the conditions of this theorem can be found if
$\scod C' \ge \scod C + 2$ whenever $C' \subset \overline C$, $C' \ne C$.
In particular, if every orbit closure $\overline C$ has even staggered
codimension in $X$, then the perversity defined by
\[
p(C) = \textstyle \frac{1}{2} \scod C,
\]
called the \emph{middle perversity}, gives rise to a self-dual, finite-length category $\p\sm X$.

\section{Schemes with Finitely Many $G$-orbits}
\label{sect:finite}

In order to study examples of staggered sheaves, one must of course be able to produce explicit $s$-structures.  Most of the axioms are quite easy to check, but axiom~\sref{ax:ext2} appears quite intimidating at first glance.  Fortunately, as we will see in this section, if $G$ acts on $X$ with finitely many orbits, it is not necessary to check this axiom.  (We retain the assumption from Section~\ref{sect:ic} that $X$ is a scheme of finite type over a field $\Bbbk$, and that $G$ is a linear algebraic group over $\Bbbk$.)

We first prove a criterion for adhesiveness in a particularly easy case.

\begin{prop}\label{prop:adhesive}
Let $Z \subset X$ be a reduced closed subscheme with an almost $s$-structure, and assume that axiom~\aref{ax:ideal} of Definition~\ref{defn:adh} holds.  If $Z$ is a union of finitely many closed $G$-orbits, then $Z$ is adhesive, and its almost $s$-structure is automatically an $s$-structure.
\end{prop}
\begin{proof}
Because $Z$ is a union of finitely many closed $G$-orbits, it has no nonempty proper dense open ($G$-invariant) subschemes.  It follows that $\tcgg Zw = \cgg Zw$ for each $w$, and hence that $\cgg Zw$ is closed under quotients.  The same applies to any other subscheme structure on $\uZ$, so we conclude that $\czsupp_{\ge w}$ is closed under quotients, and is therefore a Serre subcategory of $\cg X$.
  
Objects in the category $\qg X$ of $G$-equivariant quasicoherent sheaves
have injective hulls, by~\cite[\S II.7]{har}.  Moreover, the injective hull
of any sheaf supported on $\uZ$ must
itself be supported on $\uZ$.  Consider the
category 
\[
\qzsupp_{\ge w} = \left\{ \cF \in \qg X \,\bigg|
\begin{array}{c}
\text{every coherent subsheaf} \\
\text{of $\cF$ is in $\czsupp_{\ge w}$}
\end{array}
\right\}.
\]
Note that all sheaves in $\qzsupp_{\ge w}$ are supported on $\uZ$. 

We show now that $\qzsupp_{\ge w}$ is stable under quotients.  Given a
sheaf
$\cF \in \qzsupp_{\ge w}$, let $\cG$ be a quotient of $\cF$.  Every
coherent subsheaf of $\cG$ is a quotient of some coherent subsheaf of
$\cF$.  The latter is in $\czsupp_{\ge w}$, and since that category
is closed under quotients, we see that $\cG \in \qzsupp_{\ge w}$. 

Next, given $\cF \in \qzsupp_{\ge w}$, let $\cI$ be its injective
hull.  We claim that $\cI \in \qzsupp_{\ge w}$.  Indeed, suppose
$\cI$ had coherent subsheaf $\cH$ not in $\czsupp_{\ge w}$.  $\cH$
is supported on some subscheme structure $i': Z' \hto X$ on $\uZ$.  Say $\cH \simeq i'_*\cH_1$.  Then $\cH_1
\notin \cgg{Z'}{w}$, so the sheaf $\sigma_{\le w-1}\cH_1 \in \cgl{Z'}{w-1}$
must be nonzero.  Consider the sheaf $\cH_2 = i'_*(\sigma_{\le
w-1}\cH_1)$.  We may regard $\cH_2$ as a subsheaf of $\cI$.  Because $\cI$ is
an injective hull of $\cF$, we know that $\cH_2 \cap \cF$ is nonzero.  But
then $\cH_2 \cap \cF$ is a coherent subsheaf of $\cF$ that is not in
$\czsupp_{\ge w}$, contradicting the assumption that $\cF \in
\qzsupp_{\ge w}$. 

From the two preceding paragraphs, we see that every sheaf in $\qzsupp_{\ge w}$ admits an injective resolution all of whose terms are in
$\qzsupp_{\ge w}$.  In particular, given a coherent sheaf $\cG \in
\czsupp_{\ge w}$, we see that $\cG$ has an injective resolution by
sheaves in $\qzsupp_{\ge w}$.  Now, suppose $\cF \in \czloc {w-1}$. 
Given a sheaf $\cI \in \qzsupp_{\ge w}$, we see that
$\Hom(\cF,\cI) = 0$ (because the image of any nonzero morphism $\cF
\to \cI$ would be a coherent subsheaf of $\cI$ not in $\czsupp_{\ge
w}$).  It follows that $\Ext^r(\cF,\cG) = 0$ for all $r \ge 0$.  The same
argument shows that $\Ext^r(\cF|_V,\cG|_V) = 0$ for any open subscheme
$V$.  (Note that any such $V$ must either contain $Z$ or be disjoint from
$Z$, since $Z$ is a single $G$-orbit.)  Thus, $\gExt(\cF,\cG) = 0$.  This
establishes axiom~\aref{ax:adh-ext2} for $Z$ as a
subscheme of $X$, and it also establishes axiom~\sref{ax:ext2} for $Z$ itself.
\end{proof}

The following result gives the conditions that we will actually use in examples to build $s$-structures.

\begin{thm}\label{thm:adhesive}
Suppose $G$ acts on $X$ with finitely many orbits.  For each orbit $C \subset X$, let $\cI_C \subset \cO_X$ denote the ideal sheaf corresponding to the closed subscheme $i_C: \overline C \hto X$.  Suppose each orbit is endowed with an almost $s$-structure, and that the following two conditions hold:
\begin{enumerate}
\renewcommand{\labelenumi}{(F\arabic{enumi})}
\item For each orbit $C$, $i_C^*\cI_C|_C \in \cgl C0$.\label{f:ideal}
\item Each $\cF \in \cgl Cw$ admits an extension $\cF_1 \in \cg {\overline
C}$ whose restriction to any smaller orbit $C' \subset \overline C$ is in
$\cgl{C'}w$.\label{f:extend}
\end{enumerate}
Then there is a unique $s$-structure on $X$ whose restriction to each orbit is the given almost $s$-structure.
\end{thm}
\begin{proof}
Let $i: Z \hto X$ be a reduced closed subscheme.  By induction on the number of orbits in $Z$, we will simultaneously prove that $Z$ admits a unique $s$-structure compatible with those on the orbits it contains, and that $Z$ is adhesive as a subscheme of $X$.  Let $\cI_Z \subset \cO_X$ be its ideal sheaf.  If $Z$ consists of a single orbit, then condtion~\fref{f:ideal} is equivalent to condition~\aref{ax:ideal} of Definition~\ref{defn:adh}, and then Proposition~\ref{prop:adhesive} tells us that $Z$ is adhesive.  Otherwise, choose an open orbit $C_0 \subset Z$. Let $\kappa: Y \hto Z$ be the complementary reduced closed subscheme.  Since $Y$ contains fewer orbits, it has by assumption a unique $s$-structure compatible with those on its orbits, and it is adhesive as a subscheme of $Z$ or $X$.  Condition~\fref{f:extend} for $C_0$ is none other than the condition~\eqref{it:lower-extend} in the gluing theorem, Theorem~\ref{thm:glue}.  Invoking that theorem, we find that $Z$ admits a unique $s$-structure compatible with those on $C_0$ and on $Y$, and hence with the $s$-structures on the various orbits in $Y$. 

Next, because $Y$ is adhesive as a subscheme of $X$, its ideal sheaf $\cI_Y \subset \cO_X$ lies in $\cloc XY0$.  The latter category is Serre by Proposition~\ref{prop:cloc-serre}, and since $\cI_Z \subset \cI_Y$, we have $\cI_Z \in \cloc XY0$, or $\kappa^*i^*\cI_Z \in \cgl Y0$.  On the other hand, because $C_0$ is open in $Z$, we have $i^*\cI_Z|_{C_0} \simeq i_{C_0}^*\cI_{C_0}|_{C_0} \in \cgl {C_0}0$.  It then follows from Theorem~\ref{thm:glue} that $i^*\cI_Z \in \cgl Z0$, so condition~\aref{ax:ideal} holds for $Z$.  Next, let $V \subset X$ be the open set obtained by taking the union of $X \ssm Z$ and all orbits that are open in $Z$.  Then $V \Subset_Z X$, and $V \cap Z$ is a union of orbits that are closed in $V$, so by Proposition~\ref{prop:adhesive}, $V \cap Z$ is adhesive in $V$.  Since~\aref{ax:adh-ext2} is a condition on dense open subschemes, the fact that it holds for the closed subscheme $V \cap Z \subset V$ implies that it holds for $Z \subset X$.  Thus, $Z$ is adhesive.
\end{proof}

\section{Examples}
\label{sect:examples}

\subsection{Trivial $s$-structure}

One can always define an $s$-structure on $X$ by declaring all coherent
sheaves to be pure of step $0$.  In other words,
\[
\cgl Xw =
\begin{cases}
\cg X & \text{if $w \ge 0$,} \\
\{0\} & \text{if $w < 0$,}
\end{cases}
\qquad
\cgg Xw =
\begin{cases}
\cg X & \text{if $w \le 0$,} \\
\{0\} & \text{if $w > 0$.}
\end{cases}
\]
It is trivial to check that these categories satisfy the axioms for an
$s$-structure.  Evidently, the altitude of any irreducible closed subscheme $Z \subset X$
is $0$, so $\scod Z = \cod Z$.  The staggered $t$-structure then reduces to Deligne's perverse coherent
$t$-structure as described in~\cite{bez:pc}.

\subsection{$\C^\times$-equivariant sheaves on $\C$}

Let $X = \C$, and let $G = \C^\times$ act on $X$ by multiplication.  This
action has two orbits: $U = \C \ssm \{0\}$ and $Z = \{0\}$.  Let $j: U
\hto X$ and $i: Z \hto X$ be the inclusion maps.  On $U$, the
$G$-stabilizer of any point is trivial, so any $G$-equivariant sheaf is
free.  Let us endow $U$ with the trivial $s$-structure.

On $Z$, on the other hand, the category of equivariant coherent sheaves is
equivalent to the category of finite-dimensional representations of $G$. 
Henceforth, we will freely pass back and forth between these categories,
regarding representations as sheaves and vice versa.  Let $V_n$ denote the
$1$-dimensional representation of $G$ in which each $t \in G$ acts by the
formula $t \cdot v = t^nv$ for all $v \in V_n$. The category of
finite-dimensional $G$-representations is semisimple, and we have
\begin{equation}\label{eqn:cx-rep}
V_n \otimes V_m \simeq V_{n+m}
\qquad\text{and}\qquad
\cHom(V_n,V_m) \simeq V_{m-n}.
\end{equation}
Now, every representation $V$ has a decomposition
\[
V \simeq \bigoplus_{n \in \Z} V_n \otimes E_n,
\qquad\text{where $E_n = \Hom_G(V_n,V)$,}
\]
and in which finitely many $E_n$ are nonzero.  We define an $s$-structure
by
\begin{align*}
V \in \cgl Zw &\qquad\text{if $E_n = 0$ for all $n > w$,} \\
V \in \cgg Zw &\qquad\text{if $E_n = 0$ for all $n < w$.}
\end{align*}
It is readily seen that $\cgl Zw$ and $\cgg Zw$ are both Serre
subcategories of $\cg Z$.  Note also that all higher $\Ext$'s vanish in
$\cg Z$.  In view of the isomorphisms~\eqref{eqn:cx-rep}, we see that
these categories do indeed define an $s$-structure.

We would now like to invoke Theorem~\ref{thm:adhesive} to obtain an $s$-structure
on all of $X$.  Let $A$ be the ring $\C[x]$.  We henceforth identify $X$
with $\Spec A$, and we will freely regard $A$-modules as coherent sheaves
on $X$.  Now, the ideal $I \subset A$ corresponding to $Z$ is the
principal ideal $(x)$, and the sheaf $i^*I$ can be identified with the
vector space $(x)/(x^2)$.  An element $t \in G$ acts on $A$ by the
formula $t \cdot f(x) = f(t^{-1} x)$, so we have $i^*I \simeq V_{-1}$. 
Thus, condition~\fref{f:ideal} of Theorem~\ref{thm:adhesive} is satisfied.  Condition~\fref{f:extend} holds as well: since every object of $\cg U$ is free
and pure of step $0$, it suffices to observe that $i^*A \simeq V_0 \in
\cgl Z0$.  So that theorem applies, and we get an $s$-structure on
$X$.

$A$ itself is a dualizing object in $\cg X$.  Let us determine the
altitudes of $U$ and $Z$ with respect to this choice of dualizing complex. 
Obviously $\alt U = 0$.  Now, $\omega_Z \simeq Ri^\flat A$ satisfies
\[
i_*\omega_Z \simeq \cRHom(i_*\cO_Z, A) \simeq \cRHom(A/I, A).
\]
There is an obvious free resolution of $A/I$ given by
\[
\cdots \to 0 \to xA \to A \to 0.
\]
$\cHom(xA,A)$ is isomorphic (as an $A$-module with a $G$-action) to
$x^{-1}A$. Thus, $\cRHom(A/I,A)$ is represented by the complex
\[
0 \to A \to x^{-1}A \to 0 \to \cdots.
\]
From this complex, we see that $H^0(\omega_Z) = 0$, and $H^1(\omega_Z)
\simeq x^{-1}A/A \simeq V_1$.  We conclude that $\cod Z = \alt Z = 1$.

We therefore have $\scod U = 0$ and $\scod Z = 2$.  Since these differ by
$2$, we can invoke the results of Section~\ref{sect:ic}.  In particular, the middle perversity, given by
\[
p(U) = 0,
\qquad p(Z) = 1
\]
gives rise to a self-dual, finite-length category $\p\sm X$.  The simple objects are
\[
\cIC(X,\cO_U) \simeq \cO_X
\qquad\text{and}\qquad
\cIC(Z,V_n[n-1]) \simeq i_*V_n[n-1],
\]
where $n$ ranges over all integers.  The sheaf $x^{-1}A$ is an example of a nonsemisimple object of $\p\sm X$.  The following short exact sequence\footnote{This particular short exact sequence was suggested to me by E.~Vasserot as a desideratum for a well-behaved $t$-structure on $\cD_{\C^\times}^{\mathrm{b}}(\C)$.  This entire paper was essentially written in an attempt to give meaning to this sequence.} shows how it is built up out of simple objects:
\[
0 \to \cIC(X,\cO_U) \to x^{-1}A \to \cIC(Z,V_1) \to 0.
\]

\begin{rmk}
The axioms in Definition~\ref{defn:s} all involve conditions that are only required to hold over some dense open subscheme.  Many statements and arguments in this paper could be simplified if one could assume that those conditions held over the whole scheme (as one can if the scheme is a single $G$-orbit).

However, even in the small example considered here, some of those conditions fail over the whole scheme. For instance, $\cgg X0$ is not closed under quotients: the module $x^2A$ belongs to $\cgg X0$, but its quotient $(x^2)/(x^3) \simeq i_*V_{-2}$ does not.  Similarly, it is not true that $\gExt^1(i_*V_{-1},x^2A) = 0$: one can calculate that $Ri^\flat(x^2 A) \simeq V_{-1}[-1]$ much as we calculated $\omega_Z$ above, and then one has $\Ext^1(i_*V_{-1},x^2A) \simeq \Hom(V_{-1}, Ri^\flat x^2A[1]) \simeq \Hom(V_{-1},V_{-1}) \simeq \C$.
\end{rmk}

\subsection{Borel-equivariant sheaves on the flag variety for $SL_2(\C)$}

Now, let $G$ be the group of all matrices of the form
$(\begin{smallmatrix} c & d \\ 0 & c^{-1}\end{smallmatrix})$ with $c \in
\C^\times$, $d \in \C$, and let $G$ act on $X = \CP^1$ by the formula
\[
\begin{pmatrix}
c & d \\ 0 & c^{-1}
\end{pmatrix}
\cdot (a : b) = (ca + db : c^{-1}b).
\]
$G$ is, of course, a Borel subgroup of $SL_2(\C)$, and $\CP^1$ can be
identified with the flag variety of $SL_2(\C)$.  Under this
identification, the action described above coincides with the obvious action of a Borel
subgroup on a flag manifold.  There are two orbits: the singleton
$Z = \{(1:0)\}$, and the open subscheme $U = \{(a:b) \mid a \ne 0\}$.

Let $A$ denote the ring $\C[x,y]$.  Let us make $A$ into a graded ring by
declaring both $x$ and $y$ to be homogeneous elements of degree $1$.  Then
we may identify $X$ with $\Proj A$.  $G$ acts on $A$ as follows:
\[
\begin{pmatrix}
c & d \\ 0 & c^{-1}
\end{pmatrix}
\cdot x = c^{-1}x - dy
\qquad\text{and}\qquad
\begin{pmatrix}
c & d \\ 0 & c^{-1}
\end{pmatrix}
\cdot y = cy
\]
Let $I$ denote the principal homogeneous ideal $(y) \subset A$.  This is a
$G$-stable ideal, corresponding to the closed subscheme $Z$.

The $G$-stabilizer of the point $(0:1) \in U$ is the subgroup $T \subset G$
of diagonal matrices $(\begin{smallmatrix} c & 0 \\ 0 &
c^{-1}\end{smallmatrix})$, so the category $\cg U$ is equivalent to the
category of finite-dimensional $T$-representations.  Let $V_n$ denote the
$1$-dimensional $T$-representation in which elements of $T$ act on $v \in V_n$ by
\[
\begin{pmatrix}
c & 0 \\ 0 & c^{-1}
\end{pmatrix}
\cdot v = c^n v.
\]
We define an $s$-structure on $U$ in the same way we did in the previous
section for the closed orbit of $\C^\times$-action on $\C$: a sheaf $\cF
\in \cg U$ belongs to $\cgl Uw$ (resp.~$\cgg Uw$) if the corresponding
$T$-representation contains no summand $V_n$ with $n > w$ (resp.~$n < w$).

Consider the element $t = x/y$ in the fraction field of $A$.  This element
has degree $0$, and we can identify $U$ with $\Spec \C[t]$.  The
$G$-action on $\C[t]$ is given by
\[
\begin{pmatrix}
c & d \\ 0 & c^{-1}
\end{pmatrix}
\cdot t = \frac{c^{-1}x - dy}{cy} = c^{-2}t - d/c.
\]
The ideal $J$ corresponding to the point $(0:1)$ is the principal ideal
$(t)$.  (Note that this ideal is $T$-stable.)  Given a $\C[t]$-module $M$
with a $G$-action (or in other words, a sheaf in $\cg U$), the
corresponding $T$-representation is given by $M/JM$.

Now, consider the twisted module $\cO_X(n) = A(n)$.  The restriction of
this sheaf to $U$ can be identified with a free $\C[t]$-module $M_n$
generated by $y^n$. The $T$-action on this module is given by
\[
\begin{pmatrix}
c & 0 \\ 0 & c^{-1}
\end{pmatrix}
\cdot y^nt^k = (cy)^n (c^{-2}t)^k = c^{n-2k} y^nt^k
\]
for any $k \ge 0$.  In particular, we see that $M_n/JM_n \simeq V_n$. 
Thus, $\cO_X(n)|_U$ is pure of step $n$.  Moreover, since every
$T$-representation is a direct sum of various $V_n$'s, we see that every
sheaf in $\cg U$ is isomorphic to a direct sum of various $\cO_X(n)|_U$'s.
 
We now turn our attention to $Z$.  The category $\cg Z$ is, of course,
equivalent to the category of finite-dimensional $G$-representations.  We
define an $s$-structure on $\cg Z$ as follows: given a $G$-representation
$V$, we regard it as a $T$-representation, and then we declare that $V \in
\cgl Zw$ if $V$ contains no summand $V_n$ with $n < -w$.  Note that this
is opposite to the convention used on $U$.  Similarly, we declare that $V
\in \cgg Zw$ if $V$ contains no summand $V_n$ with $n > -w$.  As a
$T$-representation, $i^*I \simeq (y)/(y^2)$ is isomorphic to $V_1$, so
$i^*I \in \cgl Z{-1}$.  Thus, condition~\fref{f:ideal} of Theorem~\ref{thm:adhesive} holds.

Next, we check condition~\fref{f:extend}.  Since every
sheaf in $\cg U$ is a direct sum of various $\cO_X(n)|_U$'s, it suffices
to show that $i^*\cO_X(n) \in \cgl Zn$.  The submodule of $A/I$
consisting of all homogeneous elements of degree $n$ can be identified
with the vector space
\[
\C\{x^n, x^{n-1}y, \cdots, y^n\}/\C\{x^{n-1}y, x^{n-2}y^2, \cdots, y^n\}.
\]
This is clearly isomorphic to $V_{-n}$ as a $T$-representation, so we
conclude that $i^*\cO_X(n) \in \cgl Zn$, as desired.  Theorem~\ref{thm:adhesive} now gives us an $s$-structure on all of $X$.

Finally, let us compute the altitudes of $U$ and $Z$.  We take $A$ itself
as the dualizing complex.  Evidently, we have $\alt U = 0$.  A calculation
very similar to that carried out in the preceding section shows that
$H^0(\omega_Z) = 0$ and $H^1(\omega_Z) \simeq yA/A \simeq V_{-1}$.  We
conclude that $\alt Z = 1$.

It can be checked that $\scod U = 0$ and $\scod Z = 3$, so from the results
of Section~\ref{sect:ic}, we find that the staggered $t$-structure on $\dgb X$ with respect to the perversity $p(U) = 0$, $p(Z) = 1$ has a finite-length heart.  The simple objects are
\[
\cIC(X,\cO_X(n)|_U)
\qquad\text{and}\qquad
\cIC(Z, V_n[-n-1]),
\]
where $n$ ranges over all integers.


\end{document}